\documentclass[a4paper,10pt,reqno]{amsart}

\usepackage{ifpdf}
\ifpdf 
    \usepackage[pdftex]{graphicx}   
    \usepackage[pdftex,     
            plainpages=false,   
            breaklinks=true,    
            colorlinks=true,
            pdftitle=My Document
            pdfauthor=My Good Self
           ]{hyperref} 
\else 
    \usepackage{graphicx}       
    \usepackage{hyperref}       
\fi 

\usepackage{amsfonts,amsmath}	
\usepackage{amssymb}
\usepackage{verbatim}
\usepackage{amsopn}
\usepackage[english]{babel}
\usepackage{amsthm}
\usepackage{enumerate}
\usepackage{mathrsfs}	
\usepackage{mathtools}
\usepackage{color}
\usepackage{bbm}

\date{\today}

\newcounter{assu}

\theoremstyle{definition} \newtheorem{definition}{Definition}[section]
\theoremstyle{definition} \newtheorem{remark}[definition]{Remark}
\theoremstyle{plain} \newtheorem{lemma}[definition]{Lemma}
\theoremstyle{plain} \newtheorem{proposition}[definition]{Proposition}
\theoremstyle{plain} \newtheorem{theorem}[definition]{Theorem}
\theoremstyle{plain} \newtheorem{corollary}[definition]{Corollary}
\theoremstyle{definition} 
\theoremstyle{plain} 

\newtheorem{theorem01}[assu]{Theorem}
\newtheorem{theorem02}[assu]{Theorem}
\newtheorem{theorem03}[assu]{Theorem}
\newtheorem{theorem04}[assu]{Theorem}

\DeclareMathOperator*{\conv}{conv}

\DeclareMathOperator{\sign}{sign}
\DeclareMathOperator{\dist}{dist}

\DeclareMathOperator{\graph}{Graph}

\newcommand{\R}{\mathbb{R}}
\newcommand{\Q}{\mathbb{Q}}
\newcommand{\N}{\mathbb{N}}
\newcommand{\Z}{\mathbb{Z}}
\newcommand{\TV}{\text{\rm Tot.Var.}}

\DeclareMathOperator{\BV}{BV}
\DeclareMathOperator{\SBV}{SBV}

\newcommand{\I}{\textbf{I}}
\newcommand{\Id}{\mathbbm{I}}

\newcommand{\K}{\mathcal{K}}
\newcommand{\e}{\varepsilon}

\newcommand{\loc}{\mathrm{loc}}
\newcommand{\lip}{\mathrm{Lip}}
\newcommand{\supp}{\mathrm{supp}\,}
\newcommand{\diff}{\mathrm{diff}}
\newcommand{\jump}{\mathrm{jump}}

\newcommand{\Graph}{\mathrm{Graph}}

\newcommand{\X}{\mathtt{X}}
\newcommand{\U}{\mathtt{u}}

\numberwithin{equation}{section} 

\title[Structure of $L^\infty$-entropy solutions]{On the structure of $L^\infty$-entropy solutions to scalar conservation laws in one-space dimension}

\author{S. Bianchini, E. Marconi}

\address{SISSA, via Bononmea 265, I-34136 Trieste (ITALY)}

\thanks{The authors thank the CMSA at Harvard where part of this work has been written. This research has been partially supported by MIUR PRIN project nr. 2012L5WXHJ.}

\begin{document}

\begin{abstract}
We prove that if $u$ is the entropy solution to a scalar conservation law in one space dimension, then the entropy dissipation is a measure concentrated on countably many Lipschitz curves. This result is a consequence of a detailed analysis of the structure of the characteristics. \\
In particular the characteristic curves are segments outside a countably 1-rectifiable set and the left and right traces of the solution exist in a $C^0$-sense up to the degeneracy due to the segments where $f''=0$.

We prove also that the initial data is taken in a suitably strong sense and we give some counterexamples which show that these results are sharp.
\end{abstract}

\maketitle

\begin{center}
Preprint SISSA 43/2016/MATE
\end{center}

\section{Introduction}
\label{S_intro}

We consider the following problem: let $u$ be a bounded entropy solution to the scalar conservation law
\begin{equation}
\label{E_scala_cons_law_intro}
u_t + f(u)_x = 0, \qquad u \in [-M,M], \ f : \R \to \R \ \text{smooth},
\end{equation}
with initial datum $u_0(x)$. Being an entropy solution, by definition for all convex entropies $\eta$ it holds in distributions
\begin{equation}
\label{E_entropy_diss_intro}
\eta(u)_t + q(u)_x \leq 0,
\end{equation}
where $q'(u) = f'(u) \eta'(u)$ is the entropy flux. In particular the r.h.s. of \eqref{E_entropy_diss_intro} is a negative locally bounded measure $\mu$, with the additional property that $\mu(B) = 0$ for all Borel sets $B$ such that $\mathcal H^1(B) = 0$: this last property is a consequence of being the divergence of an $L^\infty$ vector field.

For BV solutions, Volpert's formula together with the definition of the entropy flux $q$ gives that
\begin{equation*}
\begin{split}
&\eta(u)_t + q(u)_x \\
& \quad = \eta'(u) \big( D^\mathrm{cont}_t u + f'(u) D^\mathrm{cont}_x u \big) \\
& \quad \quad + \sum_{i \in \N} \bigg\{ - \dot \gamma_i(t) \Big[ \eta(u(t,x+)) - \eta(u(t,x-)) \Big] + \Big[ q(u(t,x+)) - q(u(t,x-)) \Big] \bigg\} g_i(t) \mathcal H^1 \llcorner_{\Graph(\gamma_i)} \\
& \quad = \sum_{i \in \N} \bigg\{ - \dot \gamma_i(t) \Big[ \eta(u(t,x+)) - \eta(u(t,x-)) \Big] + \Big[ q(u(t,x+)) - q(u(t,x-)) \Big] \bigg\} g_i(t) \mathcal H^1 \llcorner_{\Graph(\gamma_i)},
\end{split}
\end{equation*}
where
\begin{enumerate}
\item $D^\mathrm{cont} u = (D^\mathrm{cont}_t u, D^\mathrm{cont}_x u)$ is the continuous part of the measure $Du$,
\item $u(t,x\pm)$ is the right/left limit of $u(t)$ at the point $x$,
\item the curves $\gamma_i$ are such that
\begin{equation*}
D^\mathrm{jump} u = \sum_i \big( u(t,x+) - u(t,x-) \big) \left( \begin{array}{c} 1 \\ - \dot \gamma_i(t) \end{array} \right) g_i(t) \mathcal H_i\llcorner_{\Graph(\gamma_i)}, \quad  g_i(t) = \frac{1}{\sqrt{1 + |\dot \gamma_i(t)|^2}}.
\end{equation*}
\end{enumerate}
In short we will say that \emph{the entropy dissipation is concentrated}, meaning that the measure $\mu$ is concentrated on a countably $1$-rectifiable set $J$. A simple superposition argument implies that $J$ can be chosen to be independent on $\eta$.

For general $L^\infty$-entropy solutions, in the case the flux is uniformly convex, the solution is BV for all positive times due to Oleinik estimate \cite{Oleinik}
\begin{equation*}
D_x u(t) \leq \frac{\mathcal L^1}{c t},
\end{equation*}
and then the above computation applies.

For more general flux functions, in \cite{dLR} it has been proved that under the assumption that $f$ has finitely many inflection points (together with a regularity assumption on the local behavior of $f$ about an inflection point), then again the entropy is concentrated: here the set $J$ is the set where the characteristic speed $f'(u(t,x))$ jumps, which has been proved to be a BV function in \cite{Cheng}. A counterexample to the BV regularity for $f'(u)$ in the case $f'$ does not satisfies the assumptions of \cite{Cheng} is given at the end of the paper (Section \ref{Ss_counter_2}).

A short way to state the main result of this paper is the following:

\begin{theorem01}
\label{T_main_theorem_conc}
If $u$ is a bounded entropy solution of  the scalar conservation law \eqref{E_scala_cons_law_intro}, then the entropy dissipation is concentrated.
\end{theorem01}

No assumption on the flux function $f$ have been made, i.e. it can have Cantor-like sets where $f'' = 0$. However such a statement is a corollary of a detailed description of the regularity of bounded entropy solutions, description which is at the core of this paper.

The first important result is that to every entropy solutions it is possible to associate a family of Lipschitz curves $t \mapsto \gamma(t)$ covering all $\R^+ \times \R$ with associated a set of values $w$, and a time function $\mathtt T = \mathtt T(\gamma,w)$ such that $w$ is an entropy admissible boundary value on $\Graph(\gamma \llcorner_{[0,\mathtt T(\gamma,w)]})$. The set $\mathcal K \subset \mathrm{Lip}(\R^+,\R) \times \R$ made of the couples $(\gamma,w)$ together with the function $\mathtt T(\gamma,w)$ is called \emph{complete family of boundaries}: the precise definition is Definition \ref{D_bd_fam}, which contains also additional monotonicity, connectedness and regularity property of the set of boundaries. In other words, the following diagram is commutative:

\begin{figure}[h]
\hspace{-1cm} \resizebox{6cm}{3cm}{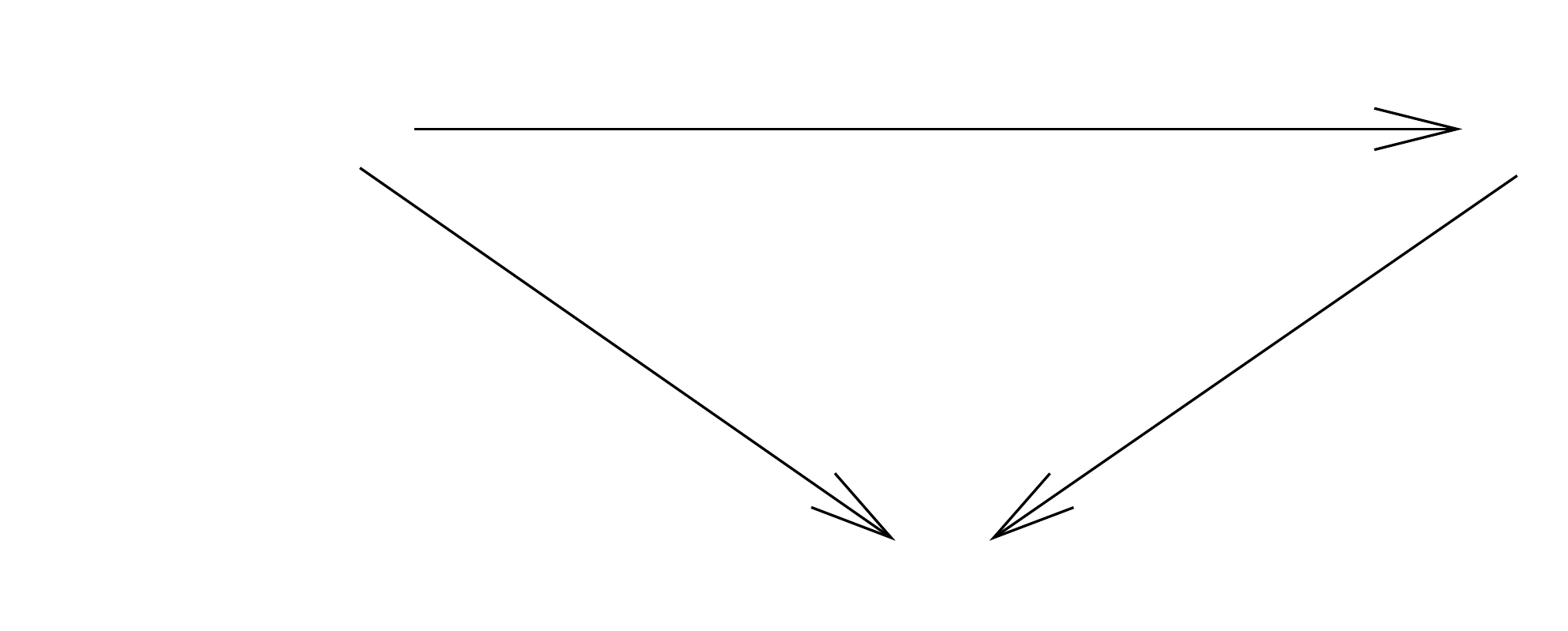}
\end{figure}

The existence of a complete family of boundaries follows from quite easy compactness arguments, due to the stability of admissible boundaries under the convergence of the boundary value and the boundary set \cite{MR996910}. The only technicality here is to prove that for a dense sets of initial data we can actually construct such a family of boundaries conditions: this is done by hand for wavefront tracking solutions, and then passed to the limit. It is interesting that the requirement to be admissible on both sides forces the curve $\gamma(t)$ to be a characteristic in the BV setting, where a suitable pointwise definition of speed is available by the Rankine-Hugoniot condition. Moreover, up to a $\mathcal H^1$-negligible set of points (and assuming for simplicity that $f$ has not flat parts), the admissible boundary values of $\gamma$ at time $t$ are equal to the segment with extremal $u(t,\gamma(t)\pm)$, as the standard entropy conditions requires. In the general case, the admissible boundary values contain the previous segment, but it may have as well parts which are in the flat part of $f$ to which $u(t,\gamma(t)\pm)$ belongs (see Lemma \ref{L_BV_case}).

There are some important properties of the set $\mathcal K$ and the function $\mathtt T$ which play a role in describing the structure of the solution $u(t)$. One almost obvious requirement is that
\begin{equation}
\label{E_complet_intor}
\Graph \ u \subset K := \Big\{ (t,x,w) : \exists (\gamma,w) \in \mathcal K, \gamma(t) = x, \mathtt T(\gamma,w) > t \Big\},
\end{equation}
i.e. there exists at least one admissible boundary for each point and the value $u(t,x)$ is one of these admissible boundary values (up to redefining $u$ in a $\mathcal L^2$-negligible set).

Consider moreover a region $\Omega$ bounded by two admissible curves $\gamma_1$, $\gamma_2$, such that
\begin{equation}
\label{E_BV_reg_intro}
\gamma_1(\bar t) = \gamma_2(\bar t) \qquad \text{and} \qquad \forall t \in (\bar t,T] \, \big( \gamma_1(t) < \gamma_2(t) \big).
\end{equation}
Then the function $u$ inside $\Omega$ solves a boundary problem where the only data are the two boundary values $w_1, w_2$ associated to $\gamma_1,\gamma_2$ respectively (we need to assume that $\mathtt T(\gamma_i,w_i) > T$, $i=1,2$, but this is not restrictive due to \eqref{E_complet_intor} above). It is a simple generalization of the construction of the Riemann solver to give explicitly the unique monotone solution in $\Omega$ (Lemma \ref{L_cov}). In particular $u \llcorner_\Omega$ is a BV function.

A second important property is that the curves $\gamma$ can be taken totally ordered:
\begin{equation*}
\exists s \in \R^+ \big( \gamma(s) < \gamma'(s) \big) \quad \Longrightarrow \quad \forall t \in \R^+ \big( \gamma(t) \leq \gamma'(t) \big).
\end{equation*}
In particular they generate a monotone flow $\mathtt X(t,y)$, with $y \in \R$ not necessarily the initial position due possible future rarefactions: denote the curve $t \mapsto \mathtt X(t,y)$ by $\gamma_y$. \\
This monotonicity allows us to define a maximal and minimal characteristic $\gamma^\pm_{t,x}$ passing through a point given $(t,x)$, and then to show that if $\gamma_{\bar y}$ is an admissible boundary and $\gamma_{\bar y}(t) \not= \gamma_y(t)$ for all $y > \bar y$, then it is a segment in $[t_1,t_2]$ (Lemma \ref{L_segments}). A symmetric result holds in the case $y < \bar y$. \\
As a corollary, it is possible to split the half plane $\R^+ \times \R$ into 4 parts, depending on the behavior of the sequences $\gamma \to \gamma^\pm_{t,x}$:
\begin{enumerate}
\item a set $A_1$ such that for all $(t,x) \in A_1$
\begin{equation*}
\gamma^-_{t,x} < \gamma^+_{t,x};
\end{equation*}
\item an open set $B$ made of regions bounded by two admissible curves satisfying \eqref{E_BV_reg_intro}: inside each component $u \in BV$;
\item a set of segments $C$, made of all admissible boundaries $\gamma_{\bar y}$ such that in $[0,\bar t]$
\begin{equation*}
\forall t \in [0,\bar t], \forall y \not= \bar y\, \big( \gamma_{\bar y}(t) \not= \gamma_y(t) \big).
\end{equation*}
\item a residual set $A_2$, where either the condition (3) above holds on one side or it is the boundary of two BV regions.
\end{enumerate}
Set $A = A_1 \cup A_2$.

The main result about this decomposition is the following (see Figure \ref{Fi_struct_ABC} and Section \ref{S_struct}):

\begin{theorem02}
\label{T_decomposition}
The exists a disjoint partition $\R^+ \times \R = A \cup B \cup C$ such that
\begin{enumerate}
\item $A$ is countably $1$-rectifiable,
\item $B$ open and $u \llcorner_B$ is locally BV,
\item $C$ is made of disjoint segments starting from $0$.
\end{enumerate}
\end{theorem02}

\begin{figure}
\resizebox{11cm}{8cm}{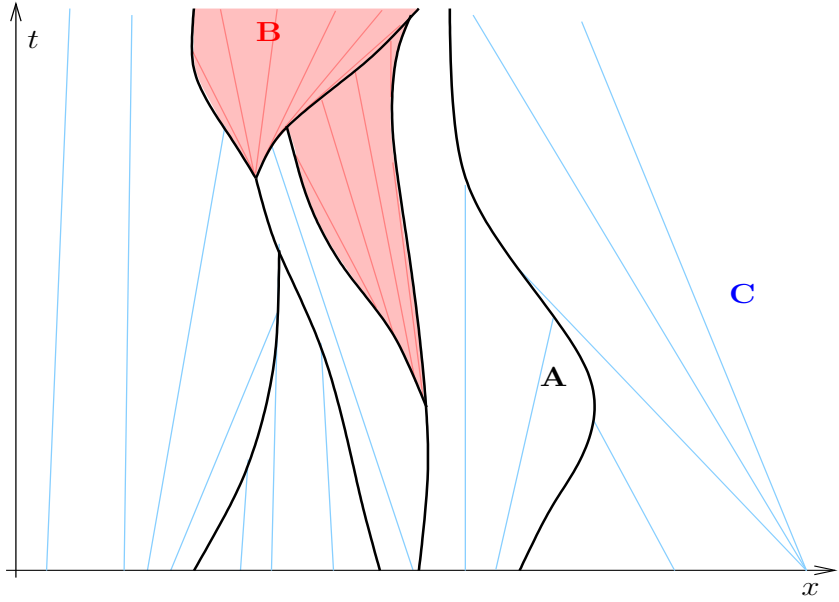}
\caption{The structure of characteristics of an entropy solution $u$.}
\label{Fi_struct_ABC}
\end{figure}

It is possible to compute the right and left limits for a given point up to the linearly degenerate components of the flux $f$: these are the connected components of the compact set $\{f'' = 0\}$. It follows that the characteristic speed has a BV structure: it is continuous outside $A$ and it has $L^1$-right/left limits across every Lipschitz curve $\gamma(t)$ up to a $\mathcal L^1$-negligible set of $t$ (Remark \ref{R_traces}). In particular we conclude that the admissible boundaries are characteristics. \\
The jump set $J$ is the set $A$ together with the countably many segments in $C$ which can dissipate, even if no characteristic is entering on both sides.

At this point one can prove Theorem 1. The only case requiring a careful analysis is the set of segments $C$. By partitioning the segments according to their length, and using the elementary fact that the slope of non intersecting segments of length $2\epsilon$ is Lipschitz w.r.t. the distance of their middle points, the conjecture on the concentration of entropy dissipation is equivalent to require that the projection of the measure $\mu$ on the middle points has no continuous part. The Dirac deltas correspond to segment of $C$ which belongs to $J$. 

\begin{figure}
\resizebox{6cm}{6cm}{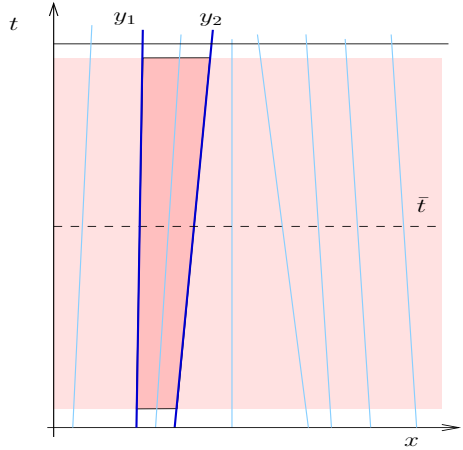}
\caption{A model set of segments parameterized by their middle point and a cylinder.}
\label{Fi_model_set_segm_intro}
\end{figure}

The basic idea is to use the two balances
\begin{equation*}
u_t + f(u)_x = 0, \qquad \eta(u)_t + q(u)_x = \mu
\end{equation*}
on the cylinders made by segments in $C$ (see Figure \ref{Fi_model_set_segm_intro}). Using the regularity of the slopes $\lambda(y)$ of the segments, which makes the bottom and top base equivalent, one concludes that the fluxes
\begin{equation}
\label{E_flux_intro}
Q(u) = q(u(y)) - \lambda(y) \eta(u(y)), \qquad F(u) = f(u(y)) - \lambda(y) u(y),
\end{equation}
are BV and Lipschitz, respectively (Lemma \ref{L_Flip}). We note that $u(y)$ can be any value on the linearly degenerate component $I$ containing $u$, because the quantities in the r.h.s. of \eqref{E_flux_intro} are constant in $I$. \\
From the balance of $F$ one recovers also that $u(t,\gamma_y(t))$ is constant for $\mathcal L^1$-a.e. $y$, and then that the following holds (Lemma \ref{L_chain_rule}): denoting with $\mathtt d_y$ the limit of incremental ratio only using segments in $C$,
\begin{equation}
\label{E_chain_rule_intro}
\mathtt{d}_y F(y) = - \mathtt{d}_y \lambda(y) u(y),
\end{equation}
which is the correct version of the smooth chain rule $(f(u) - f'(u) u)_x = - (f'(u))_x u$. \\
A geometric lemma (Lemma \ref{L_geom}) based on the assumption that if $\eta(u)/u \leq C$, then the curve
\begin{equation*}
w \mapsto \left( \begin{array}{c}
f(w) - f'(w) w \\ q(w) - f'(w) \eta(w)
\end{array} \right),
\end{equation*}
is rectifiable with tangent of bounded slope, implies that if $D_y F(y)$ has no cantor part, then $D_y Q(y)$ has no cantor part too. The chain rule \eqref{E_chain_rule_intro} above gives also that
\begin{equation*}
\mathtt{d}_y Q(y) = - \mathtt d_y \lambda(y) \eta(u(y)),
\end{equation*}
which shows that the dissipation $\mu$ has no absolutely continuous part too (Lemma \ref{L_ac_part}).

To conclude the analysis, we have to study also the set of end points of segments which do not belongs to $A$: these are starting points of shocks. The analysis on cylinders allows to include some of these points: more precisely, the points which lie in the interior of a family of cylinders shrinking to the closed segment $\gamma_y([0,T_1(y)])$, where $T_1(y)$ is the last time before $\gamma_y$ enters in $A$. Hence one can repeat the analysis above and conclude that no dissipation occurs in these points. \\
The remaining points lies on a countably $1$-rectifiable set, because $T_1(y) <  T_1(\bar y) + |y - \bar y|$ for either $y < \bar y$ or $y > \bar y$ close to $\bar y$: hence a standard criterion for rectifiability applies. One can thus use a blow up techniques, and show that the limiting solution has parallel characteristics (otherwise a shock appears) and then no dissipation is possible.

This concludes the proof of Theorem 1.

A second application of the existence of a full set of admissible boundaries is the fact that the time traces for measure valued entropy solutions are taken in a strong sense. Let
\begin{equation*}
\nu = \int \nu_{t,x} dtdx
\end{equation*}
be a Young measure on $\R^+ \times \R$ such that
\begin{equation*}
\supp \ \nu_{t,x} \subset [M,M],
\end{equation*}
and for all convex entropies $\eta$ and corresponding entropy flux $q$ it holds
\begin{equation*}
\langle \nu_{t,x},\eta \rangle_t + \langle \nu_{t,x}, q \rangle \leq 0,
\end{equation*}
where for all continuous functions $g : \R \mapsto \R$ we have used the notation
\begin{equation*}
\langle \nu_{t,x},g \rangle := \int g(v) \nu_{t,x}(dv).
\end{equation*}
Now let $\nu_{0,x}^+$ be the trace of $\nu_{t,x}$ as $t \searrow 0$. We prove the following result:

\begin{theorem03}
If $\mathfrak{d}$ is any bounded distance metrizing the weak topology on probability measure, it holds
\begin{equation*}
\lim_{t \searrow 0} \int_{\R} \mathfrak{d} \big( \nu_{t,x},\nu_{0,x}^+ \big) dx = 0.
\end{equation*}
\end{theorem03}

\noindent In particular if $\nu_{0,x}^+ = \delta_{u(t,x)}$ is a Dirac solution, we recover the fact that $t \mapsto u(t)$ is continuous in $L^1$ also at $t=0$, extending the result of \cite{MR1771520}.

The final results is that a Lagrangian representation of the solution exists:

\begin{theorem04}
There exist a monotone flow $\mathtt X$ of characteristics and a functions $\mathtt u(y)$, $\mathtt T(y)$ such that,
\begin{equation*}
u(t,\mathtt X(t,y)) = \mathtt u(y), \qquad t \leq \mathtt T(y),
\end{equation*}
for $\mathcal L^1$-a.e. $x$.
\end{theorem04}

The difference w.r.t. a complete family of boundaries is that we remove the ambiguity of the linearly degenerate parts on segments.

\subsection{Structure of the paper}
\label{Ss_structure}

The paper is organized as follows.

We first recall in Section \ref{S_prel} the basic notions of convergence of sets in the sense of Kuratowski, which gives the compactness of a complete family of boundaries. \\
It seems natural to study the problem in the setting of measure valued entropy solutions $\R^+ \times \R \ni (t,x) \mapsto \nu_{t,x}$, with $\nu_{t,x}$ a probability in $\R$: this avoids additional computations when computing traces, due to the weak compactness of Young measure (Theorem \ref{T_Young}). After recalling the definition of measure valued solution, we show indeed the existence of traces along every Lipschitz curve $\gamma$ (Proposition \ref{P_traces}), and we define the right and left admissible boundaries, Definition \ref{D_adm_bound}. The key argument for proving the existence of a complete family of boundaries is the stability of admissible boundaries w.r.t. to convergence of the parameters: the boundary curve, the boundary value, the solution $\nu_{t,x}$ and the flux function $f$ (Proposition \ref{P_stab}). \\
The last preliminary is the analysis of the Riemann problem with two boundaries: more precisely, it is the unique entropy solution in the region delimited by two Lipschitz curves starting from the same point at $t=0$. The main result is complete description of the solution, with a construction similar to the standard Riemann problem, Proposition \ref{P_V_Riem}: the main properties is the uniqueness in the family of measure valued solutions, the strict monotonicity of the solution and of the characteristic speed in an inner region.

The next section is devoted to the proof of the existence of a complete family of boundaries for measure valued solutions which are constructed as weak limits of entropy BV solution. In some sense this structure has been already constructed for BV solution (and also solutions which are piecewise continuous \cite{BiaMar1}), but the interpretation of the values $\mathtt u(y)$ along the characteristic $\mathtt X(t,y)$ as an admissible boundary value is new. This  requires to repeat the convergence analysis, starting from the front-tracking solutions,  Lemma \ref{L_wft}, and showing the stability of the set of boundaries when the flux and the solution converge weakly, Proposition \ref{P_stability}. \\
The precise definition of complete families of boundaries is given in Definition \ref{D_bd_fam}: it may seems quite strange that the only relation with the PDE is the values of the speed of the Lipschitz curves in the continuity points of the solution, but the requirement of the total ordering of the boundary curves is a strong requirement too.

Having shown the existence of a complete family of boundaries, we next consider its regularity in Section \ref{S_struct}. The result is that there are boundary traces in the strong sense, up to the flat parts of $f$. We can split the results into two parts: existence of left/right traces in $1$-d for all $t$ fixed, and existence of left/right traces in $2$-d for $L^1$-a.e. $t \in \R^+$ on a given curve. A common property is that we can take traces in the strong sense, once we quotient the real line w.r.t. the linearly degenerate components of $f$: in particular the traces are in $C^0$ if $f$ is weakly genuinely nonlinear. \\
As we noticed in the introduction, the key observation is the decomposition of $\R^+ \times \R$ into the 3 sets $A$, $B$, $C$, page \pageref{P_page_decomp}, and define from these sets the set of jumps $J$, which is easily shown to be rectifiable (Lemma \ref{L_rect}). Outside this set, the blow up converges uniformly to a linearly degenerate component. \\
At this point the regularity results are completely similar to the BV case, by just replacing the $C^0$-convergence with the uniform convergence to a linearly degenerate component: the existence of strong traces for every fixed time $t$ (actually above the minimal/maximal characteristics, Lemma \ref{L_trace_max}), the existence of strong traces outside a small cone around $\gamma$ (Proposition \ref{P_struct_J}). We note here that this is the best we can expect, due to the example 5.1.1 in \cite{BYu} where a Cantor like shock is shown.

In Section \ref{S_conc} we prove Theorem 1. After selecting a family of segments $\gamma_y([0,T])$ in $C$ which exist for a uniform time $T > 0$, and using the parameterization given by the intersection of $\gamma_y \cap \{t = T/2\}$, first we prove that the fluxes
\begin{equation*}
Q(y) := q(u(y)) - f'(u(y)) \eta(u(y)), \qquad F(y) = f(u(y)) - f'(u(y)) u(y),
\end{equation*}
are well defined, being independent on the choice of $u(y)$ in the linearly degenerate component $I(y)$: here $f'(u(y)) = \lambda(y) = \dot \gamma_y$ by the properties of the complete family of boundaries. Moreover they are BV and Lipschitz continuous w.r.t. $y$, respectively (Lemma \ref{L_Flip}). \\
A geometric lemma (Lemma \ref{L_geom}) implies that for entropies such that $\eta \leq \mathcal O(1) u$ implies that the only discontinuities of $Q(y)$ are jumps, and from the choice of jump set $J$ this possibility is ruled out. One thus deduce that for solutions there is no dissipation in the inner part of the segments composing $C$, and for measure valued solutions the disintegration is a.c., Corollary \ref{C_Cantor_part}. In particular one obtains the chain rule formulas for $F$ and $Q$, Lemma \ref{L_chain_rule} and formula \eqref{E_chain_rule}. It is also possible to represent the dissipation measure for a measure valued solution along each segment in $C$ as the derivative of the BV function $t \mapsto \int \eta(w) d\nu_{t,\gamma_y(t)}(w)$ (Lemma \ref{L_ac_part}). \\
The analysis of the endpoints of the segments in $C$ is split into two parts: either the endpoints are contained in the interior of a shrinking family of cylinders with sides in $C$, or they belong to a rectifiable set. In both cases one first prove that the disintegration has still an a.c. image measure (Lemma \ref{L_endpoints}), and then in the case of entropy solution one deduce that no dissipation occurs (Theorem \ref{P_conc}).

In Section \ref{S_init} we prove that, in a suitable sense, the initial datum is taken strongly also for measure valued solutions. The key point is that the blow up around a constant state of a measure valued entropy solution is a constant young measure, Lemma \ref{L_mv_const}. At this point the argument is quite standard: find a suitable covering (Lemma \ref{L_const_blow}), show that the limit occurs in average sense (Lemma \ref{L_init_average}), strengthen the result to have pointwise time continuity (Proposition \ref{P_init}).

The last section shows that it is possible to construct a Lagrangian representation of a measure valued entropy solution with a complete family of boundaries, hence removing the ambiguity of the value $w$ on the linearly degenerate components, Proposition \ref{P_lagr_repr}. \\
Next we give two counterexamples: the first one proves that there is a set of positive measure which is not a starting point of segments, implying that the Lagrangian representation does not allow the reconstruction of the initial data by just tracing back the value function $u(t,x), t>0$. \\
The second instead shows that the characteristic speed is not BV, even if the results contained in this paper show that it still enjoys a BV structure similar to the $1$-dimensional case i.e. $C^0$-continuity of the left and right traces.

\section{Preliminaries}\label{S_prel}

\subsection{Convergence of sets}
We recall the notion of Kuratowski convergence. Let $(X,d)$ be a metric space and $(K_n)_{n\in \N}$ be a sequence of subsets of $X$.
\begin{definition}
We define the \emph{upper limit} and the \emph{lower} limit of the sequence $K_n$ respectively by the formulas
\begin{equation*}
\begin{split}
\limsup_{n\rightarrow +\infty}K_n = \Big\{x\in X : \liminf_{n\rightarrow +\infty}d(x,K_n)=0\Big\}, \\
\liminf_{n\rightarrow +\infty}K_n = \Big\{x\in X : \limsup_{n\rightarrow +\infty}d(x,K_n)=0\Big\}.
\end{split}
\end{equation*}
We say that $K_n$ converges to $K\subset X$ in the sense of Kuratowski if 
\begin{equation*}
K=\limsup_{n\rightarrow +\infty}K_n=\liminf_{n\rightarrow +\infty}K_n.
\end{equation*}
\end{definition}
Equivalently $\displaystyle{\limsup_{n\rightarrow+\infty}K_n}$ is the set of cluster points of the of sequences $x_n\in K_n$ and 
$\displaystyle{\liminf_{n\rightarrow +\infty}K_n}$ is the set of limits of sequences $x_n\in K_n$. 

A very general compactness result holds: see \cite{Beer}.
\begin{theorem}[Zarankiewicz]
Suppose that $X$ is a separable metric space. Then for every sequence $K_n$ of subsets of $X$ there exists a convergent subsequence in the sense of Kuratowski.
\end{theorem}
We remark that without any compactness assumption the Kuratowski limit may be empty.

\subsection{Measure valued solutions on bounded domains}
In this section we introduce the notion of measure valued (briefly mv) entropy solution for the scalar conservation law 
\begin{equation}\label{E_cl}
u_t+f(u)_x=0.
\end{equation}
Here the flux $f$ is smooth and $u:\R^+\times \R\rightarrow \R$ is the conserved quantity. We follow the reference \cite{MR775191}.

\begin{definition}
 A \emph{Young measure} on $\R^+\times \R$ is a measurable map $\nu: \R^+\times\R\rightarrow \mathcal P(\R)$, in the sense that for all continuous functions $g$ on $\R$
\begin{equation*}
\langle \nu, g \rangle: (t,x) \mapsto \int g d\nu_{t,x}
\end{equation*}
is $\mathcal L^2$-measurable. A measurable function $u:\R^+\times\R \rightarrow \R$ induces the Young measure $\nu_{t,x}=\delta_{u(t,x)}$.
\end{definition}

\begin{definition}
Let $\nu^n,\nu$ be Young measures. We say that $\nu^n\rightarrow \nu$ \emph{in the sense of Young measures} if for every $g \in C_c(\R)$,
the sequence $\langle \nu^n,g\rangle$ converges to $\langle \nu,g\rangle$ with respect to the weak* topology in $L^\infty(\R^+\times \R)$.
\end{definition}

This notion is motivated by the following compactness result: see \cite{MR1036070}.
\begin{theorem}[Young]\label{T_Young}
Let $\nu^n$ be a sequence of uniformly bounded Young measures. Then there exists a subsequence $\nu^{n_k}$ and a Young measure $\nu$ such that 
$\nu^{n_k}$ converges to $\nu$ in the sense of Young measures.
\end{theorem}

\begin{definition}
We say that $(\eta,q)$ is an \emph{entropy-entropy flux pair} if $\eta:\R\rightarrow \R$ is convex and $q:\R\rightarrow \R$ satisfies $q'=\eta'f'$.
In particular we will use the following notation: for every $k\in \R$ let
\begin{equation*}
\eta^+_k(u):=(u-k)^+,\qquad \eta^-_k(u):=(u-k)^-
\end{equation*}
and the relative fluxes
\begin{equation}\label{E_flux}
 q^+_k(u):=\chi_{[k,+\infty)}(u)\big(f(u)-f(k)\big),\qquad q^-_k(u):=\chi_{(-\infty,k]}(u)\big(f(k)-f(u)\big),
\end{equation}
where $\chi_E$ denotes the characteristic function of the set $E$:
\begin{equation*}
\chi_E(u):=
\begin{cases}
1 & \mbox{if }u\in E, \\
0 & \mbox{if }u \notin E.
\end{cases}
\end{equation*}

We say that an entropy-entropy flux pair $(\eta,q)$ is a \emph{boundary entropy-entropy flux pair} with value $w$ if $\eta=\eta_k^+$ for some $k\ge w$ or 
$\eta=\eta_k^-$ for some $k\le w$ and $q$ is the corresponding flux defined by  \eqref{E_flux}.
\end{definition}

In this framework the natural notion of solution for \eqref{E_cl} is
\begin{definition}\label{D_mv_sol}
 A bounded Young measure $\nu$ is a \emph{mv entropy solution} of \eqref{E_cl} if for all entropy-entropy flux pairs $(\eta,q)$ it holds
\begin{equation}\label{E_mv_sol}
\mu:= \partial_t \langle \nu, \eta \rangle + \partial_x \langle \nu, q\rangle \leq 0
\end{equation}
in the sense of distributions. We will say that $\nu$ is a \emph{Dirac entropy solution} of \eqref{E_cl} if $\nu_{t,x}=\delta_{u(t,x)}$ where $u$ is an entropy solution of \eqref{E_cl}.
\end{definition}

\begin{remark}
In Definition \ref{D_mv_sol} we require that $\nu$ is bounded so that \eqref{E_mv_sol} makes sense for every $(\eta,q)$.
\end{remark}

We observe that \eqref{E_mv_sol} implies that the vector field $(\langle\nu, \eta\rangle, \langle \nu, q\rangle)$ is a divergence measure field, in particular it has normal traces in the sense of Anzellotti \cite{MR750538}. The argument in the next proposition is taken from \cite{MR996910}.

\begin{proposition}\label{P_traces}
Let $\nu$ be a mv entropy solution of \eqref{E_cl} and $\gamma:[0,+\infty)\rightarrow \R$ be a Lipschitz curve. Denote by $\Omega^-$ the set
\begin{equation*}
\Omega^-:=\big\{(t,x)\in \R^+\times\R : x < \gamma(t)\big\}.
\end{equation*}
Then there exists a Young measure $\nu^-:\R^+\rightarrow \mathcal P(\R)$ such that for every Lipschitz $\varphi$
with compact support and every entropy-entropy flux pair $(\eta,q)$ it holds
\begin{equation*}
\int_{\Omega^-}\varphi d\mu + \int_{\Omega^-}\varphi_t\langle \nu,\eta\rangle + \varphi_x\langle \nu,q\rangle dxdt = \int_0^{+\infty}\left(-\dot\gamma (t)\langle \nu^-_t,\eta\rangle + \langle \nu^-_t,q\rangle\right)\varphi(t,\gamma(t))dt.
\end{equation*}
The same result (with a minus sign on the r.h.s.) holds on $\Omega^+:=\{(t,x)\in \R^+\times\R : x > \gamma(t)\}$ and defines the right traces $\nu^+$.
\end{proposition}

\begin{proof}
Given $\phi \in C^\infty_c(\R^+)$ consider the function
\begin{equation*}
A_\phi(\e):=\int_{\R^+}\langle \nu_{t,\gamma(t)-\e},-\dot\gamma(t)\eta+q\rangle\phi(t)dt.
\end{equation*}
It is defined $\mathcal L^1$-a.e. by Fubini theorem.
We show that it has bounded variation: in particular it coincides $\mathcal L^1$-a.e. with his right-continuous representative $\tilde A_\phi$. 
Test the equation \eqref{E_mv_sol} with 
\begin{equation*}
\varphi(t,x)=\phi(t)\psi(\gamma(t)-x)
\end{equation*}
for some $\phi,\psi\in C^\infty_c(\R^+)$.
\begin{equation*}
\begin{split}
\int_{\Omega^-}\varphi d\mu =&~  - \int_{\Omega^-}\Big[\langle \nu_{t,x},\eta\rangle \big(\phi'(t)\psi(\gamma(t)-x) + \dot\gamma(t)\phi(t)\psi'(\gamma(t)-x)\big)-\langle \nu_{t,x},q\rangle \phi(t)\psi'(\gamma(t)-x)\Big]dx dt \\
= &~ -\int_{(\R^+)^2}\langle \nu_{t,\gamma(t)-\e}, -\dot\gamma(t)\eta + q\rangle\phi(t)\psi'(\e)dt d\e + 
\int_{(\R^+)^2}\langle \nu_{t,\gamma(t)-\e}, \eta\rangle \phi'(t)\psi(\e) dt d\e \\
= &~ -\int_{\R^+} A_\phi(\e)\psi'(\e)d\e + \int_{\R^+}\left(\int_{\R^+}\langle \nu_{t,\gamma(t)-\e},\eta\rangle \phi'(t)dt\right)\psi(\e)d\e.
\end{split}
\end{equation*}
In particular this implies that $A_\phi$ has bounded variation and 
\begin{equation*}
|DA_\phi|\leq C \TV (\phi) \mathcal L^1 + \|\phi\|_\infty p_\sharp |\mu|\llcorner \supp \varphi,
\end{equation*}
where $p:(t,x)\mapsto \gamma(t)-x$.

Let $D\subset C^1_c(\R^+)$ be a countable set dense in $C^0_c(\R^+)$.
Repeating the argument in $D$ we obtain that there exists a negligible set  $E\subset \R^+$ such that for every $\phi \in D$ and for every $\e\in \R^+\setminus E$,
\begin{equation*}
\int_{\R^+}\langle \nu_{t,\gamma(t)-\e},-\dot\gamma(t)\eta+q\rangle\phi(t)dt = \tilde A_\phi(\e).
\end{equation*}

Consider a sequence $(\e_k)_{k\in \N}\subset \R^+\setminus E$. By Theorem \ref{T_Young} there exist a subsequence $\e_{k_l}$ and a Young measure $\nu^-:\R^+\rightarrow \mathcal P(\R)$ such that for every entropy-entropy flux pair $(\eta,q)$
\begin{equation*}
\langle  \nu_{t,\gamma(t)-\e_{k_l}},-\dot\gamma(t)\eta+q\rangle \rightharpoonup \langle\nu^-_t,-\dot\gamma(t)\eta+q\rangle \quad w^*-L^\infty.
\end{equation*}
In particular for every $\phi\in D$, and using the boundedness of $\nu$ by density for every $\phi \in C^0_c(\R^+)$,
\begin{equation*}
\lim_{\e\rightarrow 0}\tilde A_\phi(\e)= \int_{\R^+}\langle \nu^-_t,-\dot\gamma(t)\eta+q\rangle\phi(t)dt.
\end{equation*}

To prove the integration by parts formula consider a test function of this form:
\begin{equation*}
(t,x)\mapsto \varphi(t,x)\psi_\e(\gamma(t)-x),
\end{equation*}
where  $\varphi$ is Lipschitz with compact support in $\R^+\times \R$ and
\begin{equation*}
\psi_\e (s)=
\begin{cases}
\frac{s}{\e} & \mbox{if }s\in (0,\e), \\
1 & \mbox{if }s\geq \e.
\end{cases}
\end{equation*}
Letting $\e\rightarrow 0$ in the divergence formula in the weak form and using $\varphi(t,\gamma(t)-\e)\rightarrow \varphi(t,\gamma(t))$, we get the claim.
\end{proof}

\begin{remark}
The fact that $ \langle\nu^-_t,-\dot\gamma(t)\eta+q\rangle$ is uniquely determined $\mathcal L^1$-a.e. for all entropy-entropy flux pairs $(\eta,q)$ implies that $\nu^-$ is uniquely determined up to regions where $f'=\dot\gamma$: more precisely let 
\begin{equation*}
O=\{u: f'(u)\neq\dot\gamma(t)\}
\end{equation*}
and $\nu^1, \nu^2$ two measures such that for all entropy-entropy flux pairs $(\eta,q)$
\begin{equation*}
\langle\nu^1,-\dot\gamma(t)\eta+q\rangle=\langle\nu^2,-\dot\gamma(t)\eta+q\rangle.
\end{equation*}
Then $\nu^1=\nu^2$ on the $\sigma$-algebra generated by $\{(u,+\infty): u\in O\}$.

It suffices to prove that $\nu^1(u^1,u^2)=\nu^2(u^1,u^2)$ for $u^1,u^2\in O$. To see this we test with an entropy $\eta_n$ such that 
\begin{equation*}
\eta_n'(u)=\frac{n}{f'(u)-\dot\gamma(t)}\left(\chi_{(u^1,u^1+\frac{1}{n})}(u)-\chi_{(u^2,u^2+\frac{1}{n})}(u)\right)
\end{equation*}
For $n$ sufficiently large this defines an entropy and we can choose the entropy flux $q_n$ such that letting $n\rightarrow +\infty$,
\begin{equation*}
\begin{split}
\eta_n\rightarrow &\frac{1}{f'(w_1)-\dot\gamma(t)}\chi_{(w_1,+\infty)}-\frac{1}{f'(w_2)-\dot\gamma(t)}\chi_{(w_2,+\infty)}, \\
q_n\rightarrow &\frac{f'(w_1)}{f'(w_1)-\dot\gamma(t)}\chi_{(w_1,+\infty)}-\frac{f'(w_2)}{f'(w_2)-\dot\gamma(t)}\chi_{(w_2,+\infty)},
\end{split}
\end{equation*}
therefore
\begin{equation*}
-\dot\gamma(t)\eta_n+q_n\rightarrow \chi_{(w^1,w^2)}.
\end{equation*}
\end{remark}

\begin{remark}
The same argument works for space-like curves. In particular a mv entropy solution $\nu$ has a representative such that for every $\bar t>0$ 
both the following limits exist in the sense of Young measures:
\begin{equation*}
 \nu_{t,x}^- = \lim_{t\rightarrow \bar t ^-}\nu_{t,x}, \qquad \nu_{t,x}^+ = \lim_{t\rightarrow \bar t ^+}\nu_{t,x} 
\end{equation*}
and they are equal for all $t$ except at most countably many.  We will denote in particular by $\nu^+_{0,x} = \lim_{t\rightarrow 0^+}\nu_{t,x}$ the trace at $t=0$.
\end{remark}
The notion of trace allows us to define in which sense a boundary condition is satisfied: see \cite{Bardos_boundary}. 

\begin{definition}\label{D_adm_bound}
A couple $(\gamma,w)$ with $\gamma:[0,+\infty)\rightarrow \R$ Lipschitz and $w\in \R$ is said to be an \emph{admissible right boundary} if for $\mathcal L^1$ a.e. t
\begin{equation}\label{E_bd_adm_cond}
\begin{split}
-\dot\gamma \langle \eta_k^+,\nu^-\rangle + \langle q_k^+,\nu^-\rangle \geq 0 & \quad \forall k\geq w, \\
-\dot\gamma \langle \eta_k^-,\nu^-\rangle + \langle q_k^-,\nu^-\rangle \geq 0 & \quad \forall k\leq w,
\end{split}
\end{equation}
where $\nu^-$ is a left trace of $\nu$ on $\gamma$. Similarly we say that  $(\gamma,w)$ is a \emph{admissible left boundary} if for $\mathcal L^1$ a.e. t
\begin{equation}\label{E_bd_adm_cond+}
\begin{split}
-\dot\gamma \langle \eta_k^+,\nu^+\rangle + \langle q_k^+,\nu^+\rangle \leq 0 & \quad \forall k\geq w, \\
-\dot\gamma \langle \eta_k^-,\nu^+\rangle + \langle q_k^-,\nu^+\rangle \leq 0 & \quad \forall k\leq w.
\end{split}
\end{equation}
We simply say that $(\gamma,w)$ is an \emph{admissible boundary} if it is admissible left and right boundary.
\end{definition}

\begin{proposition}[Stability]\label{P_stab}
Let $\nu^n$ be mv entropy solutions of \eqref{E_cl} with flux $f^n$ and $(\gamma^n,w^n)$ admissible boundaries for $\nu^n$.
Suppose that 
\begin{itemize}
\item $f^n$ are uniformly Lipschitz and $f^n\rightarrow f$ uniformly;
\item $\nu^n \rightarrow \nu$ in the sense of Young measures;
\item $w^n\rightarrow w$;
\item $\gamma^n \rightarrow \gamma$ uniformly.
\end{itemize}
Then $(\gamma,w)$ is an admissible boundary for $\nu$.
\end{proposition}

\begin{proof}
We show that the property of being an admissible right boundary is stable. Let $k<w$, $\eta=\eta_k^-$, $q^n=q^{-,n}_k$ the relative flux and
\begin{equation*}
\mu^n=\langle\nu^n,\eta\rangle_t+\langle \nu^n,q^n\rangle_x
\end{equation*}
 the dissipation. For $n$ sufficiently large we have $w^n>k$ therefore, by hypothesis, for every nonnegative test function 
$\varphi\in C^\infty_c(\R^+\times\R)$
\begin{equation}\label{E_boundary}
\int_{(\Omega^n)^-}\varphi d\mu^n+\int_{(\Omega^n)^-}\left(\varphi_t \langle \nu_{t,x}^n,\eta\rangle + \varphi_x \langle\nu_{t,x}^n,q^n\rangle \right) dxdt \geq 0.
\end{equation}
We want to pass to the limit the inequality above: 
since $\nu^n\rightarrow \nu$ in the sense of Young measures and $q^n\rightarrow q$ uniformly, by Young theorem
\begin{equation*}
\langle \nu^n,\eta \rangle \rightharpoonup \langle \nu,\eta\rangle \quad w^*-L^\infty \qquad \mbox{and} \qquad \langle \nu^n,q^n\rangle \rightharpoonup \langle \nu,q \rangle \quad w^*-L^\infty.
\end{equation*}
Moreover $\chi_{(\Omega^n)^-}\rightarrow \chi_{\Omega^-}$ strongly in $L^1$, therefore
\begin{equation}\label{E_stab1}
\int_{(\Omega^n)^-}\left(\varphi_t \langle \nu^n_{t,x},\eta\rangle + \varphi_x\langle \nu^n_{t,x,}q^n\rangle\right) dxdt \rightarrow \int_{\Omega^-}\varphi_t\langle \nu_{t,x},\eta\rangle + \varphi_x\langle \nu_{t,x},q\rangle dxdt.
\end{equation}
Let $\psi_\e\in C^\infty_c(\Omega^-)$ taking values in $[0,1]$ such that $\psi_\e(t,x) = 1$ for every $(t,x)$ such that $\dist((t,x), (\Omega^-)^c)\geq \e.$
Then, since $\mu^n$ are nonpositive
\begin{equation*}
\limsup_{n\rightarrow +\infty}\int_{(\Omega^n)^-}\varphi d \mu^n \leq \lim_{n\rightarrow +\infty}\int_{\Omega^-}\varphi \psi_\e d \mu^n = \int_{\Omega^-}\varphi \psi_\e d \mu.
\end{equation*}
Letting $\e\rightarrow 0$ we get
\begin{equation}\label{E_stab2}
\limsup_{n\rightarrow +\infty} \int_{(\Omega^n)^-}\varphi d\mu^n \leq \int_{\Omega^-}\varphi d\mu.
\end{equation}
By \eqref{E_stab1} and \eqref{E_stab2} we get that \eqref{E_boundary} holds in the limit and this is equivalent to \eqref{E_bd_adm_cond} by Proposition \ref{P_traces}.

For $\eta_k^+$ the analysis is completely analogue.
\end{proof}

DiPerna \cite{MR775191} showed that the doubling variable technique by Kruzkov \cite{Kruzhkov} applies also in the context of mv solutions: given $\nu_1,\nu_2$ mv entropy solutions of \eqref{E_cl} in an open set, it holds
\begin{equation*}
\partial_t\langle \nu_1\times\nu_2,|w - w'|\rangle + \partial_x\langle \nu_1\times\nu_2, \sign(w-w')(f(w)-f(w'))\rangle \leq 0
\end{equation*}
in the sense of distributions.
This in particular implies the uniqueness of entropy solutions in the class of mv entropy solutions for the initial value problem with $u_0\in L^\infty$.

Szepessy \cite{MR996910} extended the result to bounded domains.
\begin{proposition}\label{P_szep}
Let $T>0$ and consider a domain 
\begin{equation}\label{E_def_Omega}
\Omega = \Big\{(t,x)\in (0,T)\times \R: \gamma_1(t)<x<\gamma_2(t)\Big\}
\end{equation}
where $\gamma_1,\gamma_2:[0,T]\rightarrow \R$ are Lipschitz and $\gamma_1\leq\gamma_2$.
Let $\nu$ and $\nu'$ be two solutions of \eqref{E_cl} which satisfy the same boundary conditions $w_1$ on $\gamma_1$ and $w_2$ on $\gamma_2$. Then
\begin{equation*}
F(t)=\int_{\gamma_1(t)}^{\gamma_2(t)} \langle \nu_{t,x} \times \nu'_{t,x},|u-v| \rangle dx
\end{equation*}
has non positive derivative in the sense of distributions in $(0,T)$.
\end{proposition}
In \cite{MR2006606} Proposition \ref{P_szep} is proven in the case that
the curves $\gamma_1$ and $\gamma_2$ are constant: being the procedure the same in the case of moving boundaries we do not provide a proof.

%
%

Now we consider the particular case where $\gamma_1,\gamma_2$ are $L$-Lipschitz, $\gamma_1(0)=\gamma_2(0)$, $\gamma_1<\gamma_2$ in $(0,T]$ and the boundary data are the 
constant $a$ on $\gamma_1$ and the constant $b$ on $\gamma_2$. We assume $a\le b$, being the opposite case analogue.
%
%
The solution can be expressed quite explicitly with a construction which generalizes the one for the classical Riemann problem.

\begin{lemma}\label{L_cov}
Let $\Omega$ be defined by \eqref{E_def_Omega}. For every $(\bar t,\bar x)\in \Omega$ consider the length minimization problem
\begin{equation}\label{E_cov}
\min_{\gamma\in \mathcal A_{\bar t,\bar x}}\int_0^{\bar t}\sqrt{1+\gamma'(t)^2}dt, \quad \mbox{where}
\quad \mathcal A_{\bar t, \bar x}=\Big\{\gamma\in \lip([0,\bar t]):\gamma_1\le \gamma\le \gamma_2, \gamma(\bar t)=\bar x\Big\}.
\end{equation}
For every $(\bar t,\bar x)\in \Omega$ the minimizing curve $\gamma^{\bar t,\bar x}$ in \eqref{E_cov} exists and is unique. 

Moreover the following properties hold:
\begin{enumerate}
\item for every $(\bar t,\bar x)\in \Omega$, the function $\gamma^{\bar t,\bar x}$ is $L$-Lipschitz;
\item for every $\bar t\in (0,T)$ and for every $t\in (0,\bar t)$, the map $\bar x\mapsto \gamma^{\bar t,\bar x}(t)$ is non decreasing;
\item  for every $(\bar t,\bar x)\in \Omega$ the function $\dot \gamma^{\bar t,\bar x}$ is constant on each connected component of 
\begin{equation*}
\big\{t\in(0,\bar t): \gamma_1(t)<\gamma^x(t)<\gamma_2(t)\big\};
\end{equation*}
\item the map $v:(\bar t,\bar x)\mapsto \dot \gamma^{\bar t,\bar x}(\bar t-)$ is locally Lipschitz and has bounded variation in $\Omega$.
Moreover, for every $\bar t\in (0,T)$, the function $\bar x\mapsto v(\bar t,\bar x)$ is strictly increasing;
\item for every $\bar t\in(0,T)$ of differentiability for $\gamma_1$ and $\gamma_2$,
\begin{equation*}
\lim_{\bar x\rightarrow \gamma_1(\bar t)^+} v(\bar t,\bar x)\le \dot\gamma_1(\bar t) \qquad \mbox{and} \qquad
\lim_{\bar x\rightarrow \gamma_2(\bar t)^-} v(\bar t,\bar x)\ge \dot\gamma_2(\bar t).
\end{equation*}
\end{enumerate}
\end{lemma}
\begin{proof}
We just sketch the proof.

Existence, uniqueness and Lipschitz regularity are standard. 

Point (2) follows from uniqueness and Point (3) is trivial. Observe that uniqueness
implies that for every $(\bar t,\bar x)\in \Omega$ and $t\in (0,\bar t)$
\begin{equation*}
\gamma^{\bar t,\bar x}\llcorner [0,t]= \gamma^{t, \gamma^{\bar t,\bar x}(t)}.
\end{equation*}
In particular each level sets of $v$ is the union of segments with slope $v$ and endpoints in $\partial \Omega$. Therefore $v$ is locally Lipschitz.
The strict monotonicity with respect to $\bar x$ is a consequence of minimality and then it follows that $v$ has bounded variation. Point (5) is a
consequence of minimality too.
\end{proof}

\begin{figure}
\centering
\def\svgwidth{\columnwidth}
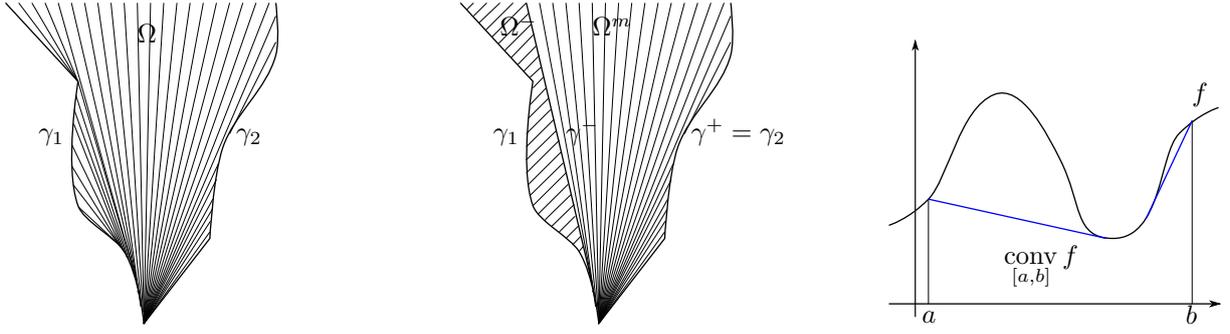
\caption{Riemann problem with boundaries: in the first figure there are the minimal curves, in the second the characteristics for the flux $f$ as in the third picture.}
\end{figure}

Denote by $\conv_{[a,b]} f:[a,b]\rightarrow \R$ the convex envelope of $f$ in $[a,b]$ and let $[\lambda^-,\lambda^+]$ the image of its derivative.
The function $(\conv_{[a,b]}f)'$ is non decreasing, we denote its pseudo-inverse by $g:[\lambda^-,\lambda^+]\rightarrow [a,b]$. 
For every $(\bar t,\bar x)\in \Omega$, define 
\begin{equation}\label{E_def_u}
u(\bar t,x)=
\begin{cases}
a &\mbox{if }v(\bar t,\bar x)\le \lambda^-, \\
g(v(\bar t,\bar x)) & \mbox{if }v(\bar t,\bar x)\in (\lambda^-,\lambda^+),\\
b & \mbox{if }v(\bar t,\bar x)\ge \lambda^+.
\end{cases}
\end{equation}
Observe that $g$ is strictly increasing (because $(\conv_{[a,b]}f)'$ is continuous) and $v(\bar t)$ is strictly increasing by the previous lemma, therefore 
$g\circ v(\bar t)$ is defined up to a countable set for $\mathcal L^1$-a.e. $\bar t$.
\begin{remark}
By the strict monotonicity of $v$ and the proof of Lemma \ref{L_cov}, each level set of $v$ is the union of at most countably many segments with 
velocity $v$ and endpoints in $\partial \Omega$. Therefore, by the strict monotonicity of $g$, the same holds for the level sets $\{u=c\}$ with $c\in (a,b)$.

\end{remark}

In the next proposition we show that $u$ is the unique solution of the boundary problem and we list some of its properties that
will be useful in the following sections.

\begin{proposition}\label{P_V_Riem}
The function $u$ defined by \eqref{E_def_u} is the unique solution of the boundary value problem in $\Omega$ in the class of mv entropy solutions.

Moreover, there exist two $L-$Lipschitz curves $\gamma^-,\gamma^+$ such that 
\begin{enumerate}[(1)]
\item for every $t\in [0,T]$, $\gamma_1(t)\le \gamma^-(t)\le \gamma^+(t) \le \gamma_2(t)$;
\item $u(t,x)=a$ for every $(t,x)$ in
\begin{equation*}
 \Omega^-:=\big\{(t,x)\in \Omega: \gamma_1(t)<x<\gamma^-(t)\big\}
\end{equation*}
and $u(t,x)=b$ for every $(t,x)$ in
\begin{equation*}
 \Omega^+:=\big\{(t,x)\in \Omega: \gamma^+(t)<x<\gamma_2(t)\big\};
\end{equation*}
\item if $\gamma_1(t)<\gamma^-(t)<\gamma_2(t)$ then $\dot\gamma^-(t)=\lambda^-$ and similarly  
	  if $\gamma_1(t)<\gamma^+(t)<\gamma_2(t)$ then $\dot\gamma^+(t)=\lambda^+$;
\item $u$ and $f'\circ u$ are strictly increasing in 
\begin{equation*}
\Omega^m:=\big\{(t,x)\in \Omega: \gamma^-(t)<x<\gamma^+(t)\big\}.
\end{equation*}
 Moreover $f'\circ u$ is locally 		Lipschitz in $\Omega^m$;
\item for $\mathcal L^1$-a.e. $t\in (0,T)$ such that $\gamma^-(t)=\gamma_1(t)$, it holds $\dot\gamma^-(t)\ge \lambda^-$;
	similarly for $\mathcal L^1$-a.e. $t\in (0,T)$ such that $\gamma^+(t)=\gamma_2(t)$, it holds $\dot\gamma^+(t)\le \lambda^+$.
\end{enumerate}
\end{proposition}
\begin{proof}
Uniqueness is a corollary of Proposition \ref{P_szep}, therefore
we need to verify that $u$ is an entropy solution in $\Omega$, $(\gamma_1,a)$ is an admissible left boundary for $u$ and 
$(\gamma_2,b)$ is an admissible right boundary for $u$. 
In the interior the analysis is the same as for the classical Riemann problem. 
Let us verify that $(\gamma_1,a)$ is an admissible left boundary for $u$, namely conditions \eqref{E_bd_adm_cond+}.
Observe that the second condition is trivial, being $u\in [a,b]$.
Denote by $u^+(t)\in \R$ the trace of $u$ in $(t,\gamma_1(t))$ for $t\in (0,T)$ and let $\bar t$ be a differentiability point of $\gamma_1$.  By 
Point (5) in Lemma \ref{L_cov} and the definition of $u$ it follows that one of the following holds:
\begin{equation*}
u^+(\bar t)=a \qquad \mbox{or}\qquad f'(u^+(\bar t))\le \dot\gamma_1(\bar t).
\end{equation*}
In the first case it is clear that \eqref{E_bd_adm_cond} is satisfied, otherwise observe that $u$ and in particular $u^+$ takes values only in the set 
$\{u:f(u)=\conv_{[a,b]} f(u)\}$. 
Therefore for every $k\in [a,u^+(\bar t)]$, it holds
\begin{equation*}
\dot\gamma_1(\bar t)\ge f'(u^+(\bar t)) = \big(\conv_{[a,b]} f\big)'(u^+(\bar t)) \ge \big(\conv_{[a,b]} f\big)'(k).
\end{equation*}
In particular \eqref{E_bd_adm_cond} is
satisfied.

In order to prove the second part of the statement, consider 
\begin{equation*}
\begin{split}
\gamma^-(t)&=\inf\big\{ \{x \in (\gamma_1(t),\gamma_2(t)): u(t,x)>a\}, \gamma_2(t)\big\},  \\
\gamma^+(t)&=\sup\big\{ \{x \in (\gamma_1(t),\gamma_2(t)): u(t,x)<b\}, \gamma_1(t)\big\}.
\end{split}
\end{equation*}
These curves are Lipschitz because they are straight lines in the open set $\{\gamma_1<\gamma^\pm<\gamma_2\}$ with bounded slope by \eqref{E_def_u}. \\
Requirements (1) and (2) are satisfied by definition of $\gamma^-$ and $\gamma^+$. Points (3), (4) and (5) follow from the respective points in
Lemma \ref{L_cov}. 
\end{proof}

\section{Lagrangian representation and complete family of boundaries}
The following result is taken from \cite{BiaMar1}.
\begin{proposition}
Let $u$ be the entropy solution of \eqref{E_cl} with initial datum $u_0\in\BV(\R)$. Then there exists a pair of functions $(\X,\U)$ such that:
\begin{enumerate}
\item $\X:[0,+\infty)\times \R\rightarrow \R$ is continuous, $t\mapsto\X(t,y)$ is Lipschitz for every $y$ and $y\mapsto \X(t,y)$ is non-decreasing for every $t$;
\item $\U:\R\rightarrow\R$ is Lipschitz;
\item the following representation formula holds: for every $t\geq 0$ and $\varphi\in C^1_c(\R)$
\begin{equation}\label{E_lagr_BV}
\int_\R u(t,x)\varphi'(x)dx = \int_\R \U(y)d D_y(\varphi\circ\X(t))(y);
\end{equation}
\item the characteristic equation holds: for every $y$, for almost every $t$
\begin{equation*}
D_t \X(t,y)=\lambda (t,\X(t,y)),
\end{equation*}
where
\begin{equation*}
\lambda(t,x)=
\begin{cases}
f'(u(t,x)) & \text{if }u(t) \text{ is continuous at }x, \\
\displaystyle \frac{f(u(t,x+))-f(u(t,x-))}{u(t,x+)-u(t,x-)} & \text{if }u(t)\text{ has a jump at }x.
\end{cases}
\end{equation*}
\end{enumerate}
\end{proposition} 

We say that $(\X,\U)$ is a \emph{Lagrangian representation} for $u$. 

Some comments are in order:
\begin{enumerate}
\item The equation \eqref{E_lagr_BV} is inspired by the following observation: let $v=u_x$ and differentiate formally \eqref{E_cl} with respect to $x$. We get
\begin{equation*}
v_t+(f'(u) v)_x=0,
\end{equation*}
which is the continuity equation with vector field $f'(u)$. The solution to that equation can be written in the following way:
\begin{equation*}
v(t) = \X(t)_\sharp v_0, \qquad \mbox{i.e.} \qquad D_x u (t) = \X(t)_\sharp (D_x u_0),
\end{equation*}
where $\X$ is the flow of the vector field $f'(u)$. This is encoded in the characteristic equation.

\item Consider the equation \eqref{E_cl} in quasilinear form: 
\begin{equation*}
u_t+f'(u)u_x=0.
\end{equation*}
We expect that $u$ is constant along characteristics of $f'(u)$, so with our notation
\begin{equation*}
u(t,x)=\U(\X(t)^{-1}(x)).
\end{equation*}
Observe that, since $\X(t)$ is continuous and non decreasing, this defines $u(t)$ on the whole $\R$ except the countable set of points
\begin{equation*}
J_t=\big\{x: |\X(t)^{-1}(x)|>1\big\}.
\end{equation*}
The set $J\subset \R^+\times\R$ such that $J\cap \{t=\bar t\}=J_{\bar t}$ is the set of discontinuity points of the solution. The regularity of $\X$
implies that it is contained in the graphs of countably many Lipschitz curves parametrized by $t$. See also below in Lemma \ref{L_rect}.
\end{enumerate}

\begin{remark}
In \cite{BiaMar1} it is shown how to pass to the limit the representation formula \eqref{E_lagr_BV} in the case when the initial datum is continuous, 
or more generally has only countably many jump-type discontinuities. Roughly speaking this follows from the fact that we can choose 
$\mathtt u^n\rightarrow \mathtt u$ uniformly and  by compactness $D_y\varphi 	\circ \X^n(t) \rightharpoonup D_y\varphi \circ \X(t)$ in the sense of 
measures. 

A similar analysis seems to be definitely more difficult for general $u_0\in L^\infty$. Indeed in this setting we can expect to represent the solution with 
$\mathtt u$ merely in $L^\infty$ and at most $\mathtt u^n\rightarrow \mathtt u$ strongly in $L^1$ for some good parametrization.
However the $\mathcal L^1$-a.e. convergence of $\U^n$ to $\U$ is not sufficient to show that $\mathtt u^n D_y\X^n(t)\rightarrow \mathtt u D_y\X(t)$ in the sense of distributions.
\end{remark}

The previous remark motivates the introduction of a more robust interpretation of Lagrangian representation. In Section \ref{S_Lagr} we will see that
 it is actually possible to recover the original Lagrangian representation from the structure of the solution.

To deal with wavefront tracking approximations, in the following definition we will consider fluxes $f$ which are $C^1$ outside finitely many points and such that
the left and right limits of $f'$ exist everywhere.
\begin{definition}\label{D_bd_fam}
Let $\nu$ be a mv entropy solution of \eqref{E_cl}. A \emph{complete family of boundaries} is a couple $(\mathcal K,T)$ where
\begin{enumerate}[(a)]
\item $\mathcal K$ is a closed subset of $\lip([0,+\infty),\R)\times \R$,
\item $T:\mathcal K\rightarrow \R$ is an upper semicontinuous function,
\end{enumerate}
and the following properties hold:
\begin{enumerate}
\item {\it monotonicity}:  $\forall (\gamma_1,w_1), (\gamma_2,w_2)\in\mathcal K$, 
\begin{equation*}
\exists \bar t: \gamma_1(\bar t)<\gamma_2(\bar t) \qquad \Longrightarrow\qquad  \forall t\geq 0, \, \gamma_1(t)\leq \gamma_2(t).
\end{equation*}
In particular there exists a total order on $\mathcal K_\gamma=\{\gamma: \exists w, (\gamma,w)\in \mathcal K\}$:
\begin{equation*}
\gamma_1 \leq \gamma_2 \qquad \Longleftrightarrow \qquad  \forall t\geq 0, \, \gamma_1(t)\leq \gamma_2(t);
\end{equation*}
\item {\it entropy admissibility}: every $(\gamma,w)\in\mathcal K$ is an admissible boundary for the mv solution $\nu$ for $t\leq T(w,\gamma)$.
Set 
\begin{equation*}
K=\Big\{(t,x,w): \exists (\gamma,w)\in \K \mbox{ such that }\gamma(t)=x, T(\gamma,w)\ge t\Big\}
\end{equation*}
and the section 
\begin{equation*}
K(t,x)=\{w: (t,x,w)\in K\};
\end{equation*}
\item {\it completeness}: 
\begin{equation*}
 \supp \left(\mathcal L^2\otimes \nu_{t,x}\right) \subset K;
\end{equation*}
\item {\it connectedness}: for all $t\geq 0$ and $\gamma_1\leq \gamma_2$, the set
\begin{equation*}
\Big\{(\gamma(t),w):(\gamma,w)\in \mathcal K, \gamma_1\leq \gamma\leq \gamma_2, t\leq T(w,\gamma)\Big\}
\end{equation*}
is connected.
\item {\it consistency with the PDE}: for every $r>0$
\begin{equation}\label{E_relation}
V_{\bar t,\bar x}(r):=\big\{\dot\gamma(t):(t,\gamma(t))\in B_{\bar t, \bar x}(r)\big\}\subset \bigcup_{w\in U_{\bar t,\bar x}(r)} \left(D^+f(w)\cup D^-f(w)\right),
\end{equation}
where $D^\pm f(w)$ denotes the super/subdifferential of $f$ at $w$ and
\begin{equation*}
U_{\bar t,\bar x}(r)=\Big\{w: \exists t,\gamma\ \Big( (\gamma,w)\in \mathcal K, (t,\gamma(t))\in \bar B_{\bar t, \bar x}(r) \text{ and }t\leq T(w,\gamma)\Big)\Big\}.
\end{equation*}
\end{enumerate}
\end{definition}
For a $C^1$ flux $f$ condition (5) reduces to 
\begin{equation*}
\big\{\dot\gamma(t):(t,\gamma(t))\in B_{\bar t, \bar x}(r)\big\}\subset \big\{f'(w):w \in U_{\bar t,\bar x}(r)\big\}.
\end{equation*}
\begin{remark}
The completeness property and the fact that $K$ is closed imply that for every $(t,x)\in \R^{+} \times \R$ the section $K(t,x)\ne \emptyset$.
\end{remark}
\begin{remark}
The monotonicity and the covering properties imply that the curve in $\K_\gamma$ can be parametrized by $\R$: more precisely there exists a monotone invertible map $p:\R\rightarrow \K_\gamma$. We will denote $p(y)$ by $\gamma_y$.
\end{remark}

\begin{proposition}\label{P_stability}
Let $\nu^n$ be mv entropy solutions of \eqref{E_cl} and let $(\K^n,T^n)$ be a complete family of boundaries for $\nu^n$. Assume that 
\begin{enumerate}
\item $f^n\rightarrow f$ uniformly and $\graph(D^+f^n) \cup \graph(D^-f^n) \rightarrow \graph(f')$ in the sense of Kuratowski,
\item $\nu^n \rightarrow \nu$ in the sense of Young measures,
\item $\K^n\rightarrow \K$ in the sense of Kuratowski,
\end{enumerate}
and set
\begin{equation*}
T(\gamma,w)= \inf_{U\in \mathcal U(\gamma,w)}\limsup_{n\rightarrow +\infty} \sup_{(\gamma',w')\in U}T^n(\gamma',w') = - \Gamma \mbox{-}\liminf_{n\rightarrow +\infty} (-T^n (\gamma,w)).
\end{equation*}
Then $(\K,T)$ is a complete family of boundaries for $\nu$.
\end{proposition}
\begin{proof}
We have to verify conditions (1) to (5) in Definition \ref{D_bd_fam}: condition (2) follows from Proposition \ref{P_stab} and conditions (1), (3) and (4)
follow from the very definition of Kuratowski convergence, convergence in the sense of Young measures and the definition of $T$. About condition (5),
each $\gamma\in \K_\gamma$ is the uniform limit of $\gamma^n\in \K^n_\gamma$, in particolar $\dot\gamma^n\rightarrow \dot\gamma$ weakly. 
Therefore, for every $r>0$,
\begin{equation}\label{E_K_stab_1}
V_{\bar t,\bar x}(r)\subset K\mbox{-}\limsup_{n\rightarrow \infty} \conv V^n_{\bar t,\bar x}(r).
\end{equation}
By the Kuratowski convergence of the complete families of boundaries 
\begin{equation*}
K\mbox{-}\limsup_{n\rightarrow \infty}U^n_{\bar t,\bar x}(r)\subset U_{\bar t,\bar x}(r),
\end{equation*}
hence, by assumption (1),
\begin{equation}\label{E_K_stab_2}
\bigcup_{w\in U_{\bar t,\bar x}(r)}\{f'(w)\} \supset K\mbox{-}\limsup_{n\rightarrow \infty}
\bigcup_{w\in U^n_{\bar t,\bar x}(r)} \left(D^+f^n(w)\cup D^-f^n(w)\right).
\end{equation}
Since $ \bigcup_{w\in U^n_{\bar t,\bar x}(r)} \left(D^+f^n(w)\cup D^-f^n(w)\right)$ is convex by the connectedness property, the claim follows from 
\eqref{E_K_stab_1} and \eqref{E_K_stab_2}.
\end{proof}

In the following lemma we construct a complete family of boundaries for wave-front tracking approximate solutions: it is standard to assume that only binary interactions among shocks occur.
For reference on the wavefront-tracking scheme see \cite{Dafermos}.

\begin{lemma}\label{L_wft}
Let $u:\R^+\times \R\rightarrow 2^{-k}\Z$ be a wave-front tracking solution of \eqref{E_cl}. Then there exists a complete family of boundaries for $u$ such that 
\begin{enumerate}
\item every $\gamma\in \K_\gamma$ is piecewise affine and for all except finitely many positive times it holds
\begin{equation*}
\dot\gamma(t)=\lambda(t,\gamma(t)),
\end{equation*}
where 
\begin{equation}\label{E_lambda}
\lambda(t,x)=
\begin{cases}
f'(u(t,x)) & \mbox{if $u$ is continuous at }(t,x), \\
\displaystyle \frac{f(u(t,x+))-f(u(t,x-))}{u(t,x+)-u(t,x-)} & \mbox{if $u$ has a jump at }(t,x);
\end{cases}
\end{equation}
\item for all $(t,x)\in \R^+\times \R$ except the cancellation points
\begin{equation*}
K(t,x)=\conv (u^-,u^+),
\end{equation*}
where $u^-$ and $u^+$ denote the left limit and right limit respectively.
At every cancellation point
\begin{equation*}
K(t,x)=\conv (u^-,u^+)\cup I,
\end{equation*}
where $I$ is the set of values of $u$ that is canceled in $(t,x)$.
\end{enumerate}
\end{lemma}

\begin{proof}
We just prove the existence because the properties (1) and (2) follow easily from the construction.
{\it Step 1}.
We first construct the candidate admissible boundaries on the set $J$ of discontinuity points of $u$. 
Consider a shock starting at $t=0$ from $\bar x$ with left and right limits $u^-$ and $u^+$ respectively.
For every $w\in \conv(u^-,u^+)\setminus 2^{-k}\Z$ consider the unique Lipschitz continuous curve $\gamma_w:I_w\rightarrow \R$ such that
for all $t\in I_w$, $w\in \conv(u(t,\gamma_w(t)-),u(t,\gamma_w(t)+))$, where $I_w=[0,\bar t]$ if the value $w$ is canceled at time $\bar t$ and
$I_w=[0,+\infty)$ if the value $w$ is not canceled.
Denote the set of pairs $(\gamma_w,w)$ by $\tilde \K^1$ and set $\widetilde{T}(\gamma_w,w)=\sup I_w$. 
The following monotonicity property holds: let $w_1<w_2$ and $t \in I_{w_1}\cap I_{w_2}$ such that $\gamma_{w_1}(t)=\gamma_{w_2}(t)=x$. 
Then $u(t,x-)<u(t,x+)$ implies $\gamma_{w_1} \le \gamma_{w_2}$ in $I_{w_1}\cap I_{w_2}$ and similarly 
$u(t,x-)>u(t,x+)$ implies $\gamma_{w_1} \ge \gamma_{w_2}$ in $I_{w_1}\cap I_{w_2}$.
The proof is by direct inspection of binary interactions of shocks.

{\it Step 2}.
Next we construct segments in $\R^+\times \R\setminus J$. For every $(\bar t,\bar x)\in \R^+\times \R\setminus J$ consider the straight line 
$\gamma^{\bar t,\bar x}:\R^+\rightarrow \R$ where $\gamma^{\bar t,\bar x}(\bar t)=\bar x$, ${\gamma^{\bar t,\bar x}}'(t)= f'(u(\bar t,\bar x))$.
In order to have monotonicity we consider $\gamma^{\bar t,\bar x}$ restricted to the connected component $(t_1(\bar t,\bar x),t_2(\bar t, \bar x))$ 
of $\{t\in \R^+:(t,\gamma^{\bar t,\bar x}(t)\in \R^+\times \R\setminus J\}$ which contains $\bar t$.
Denote the set of pairs $(\gamma^{\bar t,\bar x},u(\bar t,\bar x))$ by $\tilde \K^2$ and set 
$\widetilde{T}(\gamma^{\bar t,\bar x},u(\bar t,\bar x))=t_2(\bar t, \bar x)$.

In order to construct a complete family of boundaries we begin to extend the curves in $\tilde\K^1$ and $\tilde \K^2$ to the whole $\R^+$.

{\it Step 3}.
First, for every $(\bar t, \bar x)\in \R^+\times \R\setminus J$ we prolong $\gamma^{\bar t,\bar x}$ to $[0,t_2(\bar t,\bar x))$.
Denote by $\gamma^-$ and $\gamma^+$ the left and the right boundary of the connected component of $\R^+\times \R\setminus J$
which contains $(\bar t,\bar x)$.
If $t_1(\bar t,\bar x)>0$ at least one of the following holds:
\begin{equation}\label{E_dicot}
\gamma^-(t_1(\bar t,\bar x))=\gamma^{\bar t,\bar x}(t_1(\bar t,\bar x)+) \qquad \mbox{or}\qquad 
\gamma^+(t_1(\bar t,\bar x))=\gamma^{\bar t,\bar x}(t_1(\bar t,\bar x)+).
\end{equation}
If the first condition holds we set $\gamma^{\bar t,\bar x}=\bar \gamma$ in $[0,t_1(\bar t,\bar x)]$ where $\bar\gamma:[0,t_1(\bar t,\bar x)]\rightarrow \R$
is the unique curve such that there exists $(\gamma_w,w)\in\tilde K^1$ for which
\begin{equation*}
\gamma_w=\bar \gamma \mbox{ in } [0,t_1(\bar t,\bar x)], \quad 
|w-u(\bar t,\bar x)|< 2^{-k}, \quad 
\widetilde{T}(\gamma_w,w)>t_1(\bar t,\bar x), \quad
\gamma_w=\gamma^- \mbox{ in }(t_1(\bar t,\bar x),t_1(\bar t,\bar x)+\e)
\end{equation*}
for some $\e>0$. If the first condition in \eqref{E_dicot} does not hold the analogue extension can be done for $\bar \gamma=\gamma^+$ in a 
right neighborhood of $t_1(\bar t,\bar x)$.
This extension maintains the monotonicity. Denote by $\tilde \K^3$ this extension of $\tilde K^2$.

For the extension in the future the only constraint is the monotonicity: 
denote by $\tilde \K=\tilde\K^1 \cup\tilde \K^3$ and let $(\gamma,w)\in \tilde \K$ and $t>\widetilde{T}(\gamma,w)$. Then we consider the following extension:
\begin{equation*}
\gamma(t)=\sup_{(\gamma',w') \in \tilde \K} \Big\{\gamma' (t): \widetilde{T}(\gamma',w')\ge t \mbox{ and }\exists t'<T(\gamma,w) \mbox{ such that }\gamma'(t')<\gamma(t')\Big\}.
\end{equation*}
It is fairly easy to check that with this extension the monotonicity is preserved.

Finally let $\K$ be the closure of the family constructed above with respect the product of local uniform convergence topology and 
the standard topology on $\R$ and let $T$ the minimal upper semicontinuous extension of $\widetilde{T}$.

To conclude, we have to verify the properties in Definition \ref{D_bd_fam}. Only the entropy admissibility is not straightforward but it is a consequence of the fact that for every
$(\gamma,w)\in \K$ and $t< T(\gamma,w)$ it holds $w\in \conv (u(t,\gamma(t)-),u(t,\gamma(t)+))$.
\end{proof}

The construction of a complete family of boundaries for approximations by wave-front tracking and the stability proven in Proposition \ref{P_stability} imply the following result.
\begin{theorem}
For every entropy solution of \eqref{E_cl} with initial data $u_0\in L^\infty$ there exists a complete family of boundaries.
\end{theorem}

Since the linearly degenerate components of the flux play a significant role in what follows we introduce the following notation.
\begin{definition}\label{D_Lf}
We denote by $\mathcal L_f$ the set of maximal closed intervals on which $f'$ is constant. If there are no intervals where $f'$ is constant we say that $f$ 
is weakly genuine nonlinear. For every $w\in \R$ we denote by $I_w$ the unique element of 
$\mathcal L_f$ which contains $w$. Moreover if $I\in \mathcal L_f$, we write $f'(I)$ to indicate $f'(w)$ for some $w\in I$. Finally, when $\mathcal L_f$ is
considered as a topological space it is endowed with the quotient topology obtained from the euclidean topology on $\R$ by the relation that identifies
elements belonging to the same $I\in \mathcal L_f$.
\end{definition}

By the completeness property of the complete family of boundaries we have that $ \supp \left(\mathcal L^2\otimes \nu_{t,x}\right) \subset K$; the next lemma
is a first result about the opposite inclusion. We will see that it holds up to linearly degenerate components of the flux. 
 
\begin{lemma}\label{L_bd_const}
Let $\nu$ be a mv entropy solution on $\R^2$ such that there exists $I=[a,b]\in \mathcal L_f$ for which $\mathcal L^2$-a.e. $(t,x)\in \R^2$, $\supp\, \nu_{t,x}\subset I$ and let $(\gamma,w)$ an admissible boundary for $\nu$. Then $w\in I$.
\end{lemma}
\begin{proof}
Assume by contradiction that there exists an admissible boundary $(\gamma, w)$ with $w\notin I$ and let $\sigma= f'(I)$. First we prove that 
$\dot\gamma = \sigma$.
Without loss of generality take $w<a$. By the admissibility condition \eqref{E_bd_adm_cond} for every $w\le k\le a $
\begin{equation*}
\begin{split}
0\leq &~ \langle \nu^-, f(\lambda)-f(k) - \dot\gamma(\lambda - k)\rangle \\
= &~ \langle \nu^-, f(\lambda)-f(k)-\sigma (\lambda - k)+ (\sigma- \dot\gamma)(\lambda -k)\rangle \\
= &~ f(a)-f(k)-\sigma (a-k) + (\sigma-\dot\gamma)(\langle \nu^-,\lambda\rangle -k)
\end{split}
\end{equation*}
because $f(w)-\sigma w$ is constant on $I$.
Since $ f(a)-f(k)-f'(a)(a-k)=o(|a-k|)$ as $k\rightarrow a$, $\sigma = f'(a)$ and $\langle \nu^-,\lambda\rangle\geq a$ we get $\sigma \geq \dot\gamma$. 
The same argument on the admissibility condition from the right of $\gamma$ implies that $\sigma \leq \dot\gamma$ therefore $\sigma=\dot\gamma$.
In particular the condition above reduces to $0\leq f(a)-f(k)-\sigma (a-k)$ for all $w\le k\le a $ and the one on the right to $0\geq f(a)-f(k)-\sigma (a-k)$ for all 
$w\le k\le a $. This means that $w$ and $a$ belongs to the same linearly degenerate component of the flux and that is a contradiction by maximality of $[a,b]$.
\end{proof}

In the next lemma we state some additional properties of the complete family of boundaries when the solution has bounded total variation.
These results are based on a blow-up argument: since it will be useful later we introduce the notion in the setting of mv entropy solutions. 
\begin{definition}\label{D_blow}
Let $\nu$ be a mv entropy solution on $\R^+\times \R$. Given $(\bar t,\bar x)\in [0,+\infty)\times \R$ and $\e>0$ consider 
\begin{equation*}
\nu_{t,x}^\e=\nu_{\bar t + \e t, \bar x + \e x}
\end{equation*}
defined for $\bar t+\e t\ge 0$.
For all entropies $\eta$ the dissipation $\mu^\e$ of $\nu^\e$ is given by 
\begin{equation*}
\mu^\e(B)= \frac{1}{\e}\mu ((\bar t,\bar x) + \e B),
\end{equation*}
for a Borel set $B\subset \R^2$ such that $(\bar t,\bar x) + \e B \subset \R^+\times \R$, where $\mu$ denotes the dissipation of $\nu$.
Every limit in the sense of Young measures of $\nu^\e$ as $\e\rightarrow 0$ is called \emph{blow-up} of $\nu$ at $(\bar t,\bar x)$.
\end{definition}
It is standard to check that every blow-up is a mv entropy solution.

\begin{lemma}\label{L_BV_case}
Let $\nu$ be a mv solution of \eqref{E_cl} with a complete family of boundaries $(\K,T)$ and let $\Omega\subset \R^+\times\R$ be such that for every
$(t,x)\in \Omega$, 
\begin{equation*}
\nu_{t,x}=\delta_{u(t,x)},
\end{equation*}
where $u\in \BV(\Omega)$. Then for every $\gamma\in \K_\gamma$ and for $\mathcal L^1$-a.e. $t>0$ such that $(t,\gamma(t))\in \Omega$, it holds
\begin{equation}\label{E_speed_BV}
\dot\gamma(t)=\lambda(t,\gamma(t))
\end{equation}
where $\lambda$ is defined in \eqref{E_lambda}.
Moreover for $\mathcal H^1$-a.e. $(t,x)\in \Omega$,
\begin{equation}\label{E_K_BV}
\conv(u^-,u^+)\subset K(t,x)\subset \conv \left(I_{u^-},I_{u^+}\right),
\end{equation}
where $u^-$ and $u^+$ denote the left and right limits if $(t,x)$ is a jump point of $u$ and $u^-=u^+=a$ if $u$ has a Lebesgue point with value $a$ in $(t,x)$.
\end{lemma}
The intervals $I_{u^-}$ and $I_{u^+}$ are defined in Definition \ref{D_Lf}.
\begin{proof}
For $\mathcal H^1$-a.e. $(\bar t,\bar x)\in \Omega$ there are two possibilities for the $L^1$-blow-up of $u$ in $(\bar t,\bar x)$:
\begin{enumerate}
\item the limit is contained in $I$ for some $I\in \mathcal L_f$;
\item the limit is a jump with $u^-$ and $u^+$ which do not belong to the same $I\in \mathcal L_f$.
\end{enumerate}
In the first case \eqref{E_speed_BV} follows from Lemma \ref{L_bd_const} and \eqref{E_relation}. Moreover the first inclusion in \eqref{E_K_BV}
follows from the connectedness property in Definition \ref{D_bd_fam} and the second inclusion follows from Lemma \ref{L_bd_const}, being the blow-up a mv entropy solution.

In the second case let 
\begin{equation*}
\bar u(t,x) =
\begin{cases}
u^- & x<\lambda t, \\
u^+ & x>\lambda t,
\end{cases}
\end{equation*}
be the $L^1$-blow-up. By Lemma \ref{L_bd_const}, the speed of admissible boundaries in $\{(t,x):x<\lambda t\}$ is $f'(u^-)$ and similarly the speed of 
admissible boundaries in $\{(t,x):x>\lambda t\}$ is $f'(u^+)$. Moreover, since $\nu$ is a mv entropy solution, then
\begin{equation*}
f'(u^-)\ge\lambda\ge f'(u^+)
\end{equation*}
 and if $\gamma$ is differentiable at $\bar t$, its blow-up is a straight line. 
 So the unique velocity that $\gamma$ can have without violating 
 the monotonicity is 
$\lambda(\bar t,\gamma(\bar t))$.

As in the previous case, the first inclusion in \eqref{E_K_BV} follows from connectedness.

\noindent About the second inclusion consider an admissible
boundary $(\gamma, w)$ for $\bar u$. If $\gamma(t)\ne \lambda t$ the result follows from Lemma \ref{L_bd_const}. Finally consider the case 
$\gamma(t)=\lambda t$ and suppose without loss of generality that $w<u^-< u^+$ and let $k\in (w,u^-)$. By admissibility conditions \eqref{E_bd_adm_cond} and \eqref{E_bd_adm_cond+},
\begin{equation*}
f(u^+)-\lambda(u^+-k)\leq f(k)\leq f(u^-)-\lambda(u^--k)=f(u^+)-\lambda(u^+-k),
\end{equation*}
therefore $w\in I_{u^-}$ and $\lambda= f'(u^-)$.
\end{proof}

Now we consider the particular case of the Riemann problem with two boundaries. With the same notation as in Proposition \ref{P_V_Riem}, the previous result
implies that the complete family of boundaries is uniquely determined in $\Omega^m$. 
\begin{corollary}\label{C_uniq_K}
Let $\nu$ be a mv entropy solution of \eqref{E_cl} with a complete family of boundaries $(\K,T)$ and let $(\gamma_1,a),(\gamma_2,b)\in \K$ such that
for some $0\le t_1<t_2$:
\begin{enumerate}
\item $\gamma_1(t_1)=\gamma_2(t_1)$;
\item $\gamma_1(t)<\gamma_2(t)$ for every $t\in (t_1,t_2)$;
\item $T(\gamma_1,a)>t_2,  T(\gamma_2,b)>t_2$.
\end{enumerate}
Let $(\gamma,w)\in \K$ be such that there exists $\bar t>0$: $(\bar t,\gamma(\bar t))\in \Omega^m$ and $T(\gamma,w)\ge\bar t$. 

Then $w\in (a,b)$ and $\gamma$ coincides in $[t_1,\bar t]$ with the unique Lipschitz curve $\tilde \gamma:[t_1,\bar t]\rightarrow \R$ such that
\begin{enumerate}[(a)]
\item for all $t\in [t_1,\bar t]$, $\gamma_1(t)\le \tilde\gamma(t)\le \gamma_2(t)$;
\item it holds 
\begin{equation*}
\graph(\tilde\gamma) \cap \Omega^m = \Big\{(t,x)\in \Omega: v(t,x)=\big(\conv_{[a,b]}f\big)'(w)\Big\},
\end{equation*}
where $v$ is defined in Lemma \ref{L_cov}.
\end{enumerate}
Moreover for every $(t,x)\in \Omega^m$
\begin{equation*}
K(t,x)=\conv (u^-,u^+),
\end{equation*}
where $u^-,u^+$ denote the left and right limits of $u$ at time $t$ in $x$.
\end{corollary}
\begin{proof}
Let $(\gamma,w)\in \K$ be as in the statement. 
By Proposition \ref{P_V_Riem} point (4), it follows that for every $(t,x)\in \Omega^m$ it holds 
\begin{equation*}
\conv(u^-,u^+)=\conv(I_{u^-}, I_{u^+})
\end{equation*}
otherwise $f'\circ u$ would not be strictly increasing, therefore the last part of the statement is a consequence of \eqref{E_K_BV} and it implies that $w\in (a,b)$.

Since $v$ is strictly increasing with respect to $x$ in $\Omega^m$, the curve $\tilde\gamma$ that satisfies $(a)$ and $(b)$ in the
statement is unique. Indeed suppose by contradiction that $\gamma (\tilde t) \ne\tilde\gamma (\tilde t)$ for some $\tilde t\in  (t_1,\bar t)$, then $u$ solves the
boundary Riemann problem with boundary data equal to $w$ on $\gamma$ and $\tilde \gamma$. In particular $u(t,x)\equiv w$ for every $(t,x)$ in the open region
delimited by $\gamma$ and $\tilde \gamma$ and this contradicts the strict monotonity of $u$ in $\Omega^m$.
\end{proof}

\begin{remark}
A complete family of boundaries for a Riemann problem with two boundaries is not uniquely determined in $\Omega^-$ and $\Omega^+$ if $I_a$ and $I_b$ are
non trivial. However we will see that $\Omega^-=\Omega^+=\emptyset$ in the setting of the corollary above.
\end{remark}

\section{Structure of $\K$}
\label{S_struct}
In this section we see that a complete family of boundaries for a mv entropy solution enjoys additional properties than the ones
required in the definition. More precisely we prove that $\R^+\times \R$ is covered by characteristics which are straight lines outside a 1-rectifiable 
set of jumps, similarly to the case of solutions with bounded variations.

First we introduce some notation:
denote by $B_{t,x}(r)$ the ball in $\R^+\times \R$ of centre $(t,x)$ and radius $r$. Given $\gamma\in\mathcal K_\gamma$, a differentiability point $\bar t$ of $\gamma$ and $r,\delta>0$, let
\begin{equation*}
B^{\delta +}_{\bar t,\gamma}(r) := \Big\{(t,x)\in B_{\bar t,\gamma(\bar t)}(r) : x> \gamma (\bar t)+ \dot\gamma(\bar t)(t-\bar t) + \delta|t-\bar t|\Big\},
\end{equation*}
\begin{equation*}
B^{\delta -}_{\bar t,\gamma}(r) := \Big\{(t,x)\in B_{\bar t,\gamma(\bar t)}(r) : x< \gamma (\bar t)+ \dot\gamma(\bar t)(t-\bar t) - \delta|t-\bar t|\Big\}.
\end{equation*}
Accordingly we define 
\begin{equation*}
\begin{split}
U_{\bar t,\bar x}(r) & := \Big\{ w \in \R : \exists t\in \R^+, (\gamma,w)\in \mathcal K \mbox{ such that } T(\gamma, w )> t, (t,\gamma(t)) \in B_{\bar t,\bar x}(r)\Big\}, \\
U^{\delta \pm}_{\bar t, \bar \gamma}(r) & := \Big\{ w \in \R : \exists t\in \R^+, (\gamma,w)\in \mathcal K \mbox{ such that } T(\gamma, w )> t, (t,\gamma(t)) \in B^{\delta \pm}_{\bar t,\bar \gamma}(r)\Big\}.
\end{split}
\end{equation*}
For every $(t,x)\in \R^+\times \R$, introduce the maximal characteristic $\gamma_{t,x}^+$ as the maximum in $\K_\gamma$ of the
curves $\gamma$ such that $\gamma(t)=x$. The maximum exists being $\K$ closed. Similarly let $\gamma_{t,x}^-$ be the minimal characteristic, 
and denote by
\begin{equation*}
\begin{split}
U^+_{t,x}(r)& :=\Big\{ w\in \R : \exists t\in \R^+, (\gamma,w)\in \mathcal K \mbox{ such that } T(\gamma,w)>t,(t,\gamma(t))\in B_{t,x}(r) \mbox{ and } \gamma>\gamma^+_{t,x},  \Big\}, \\
U^-_{t,x}(r)& :=\Big\{ w\in \R : \exists t\in \R^+, (\gamma,w)\in \mathcal K \mbox{ such that } T(\gamma,w)>t,(t,\gamma(t))\in B_{t,x}(r) \mbox{ and } \gamma<\gamma^-_{t,x},  \Big\}.
\end{split}
\end{equation*}
See Figure \ref{F_cono}.

\begin{figure}
\centering
\def\svgwidth{0.5\columnwidth}
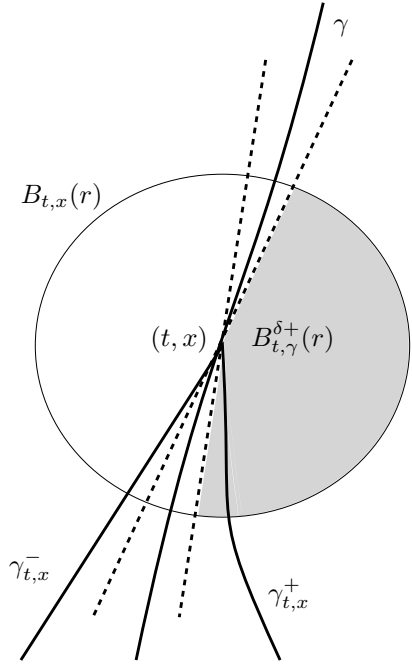\label{F_cono}
\caption{Maximal and minimal curves and cone.}
\end{figure}

\begin{lemma}\label{L_segments}
Consider $\bar \gamma \in \K_\gamma$ and $[t_1,t_2]\subset \R^+$. Suppose that $\forall \gamma \in \K_\gamma$,
\begin{equation*}
\exists\, \bar t \in [t_1,t_2] : \gamma (\bar t)>\bar \gamma(\bar t) \quad \Longrightarrow \quad \forall  t \in [t_1,t_2] : \gamma ( t)>\bar \gamma( t).
\end{equation*}
Then there exists an interval $I\in \mathscr L_f$ such that for every sequence $(\gamma^n,u^n)\in \K$ satisfying
\begin{enumerate}
\item $\gamma^n\llcorner_{[t_1,t_2]}>\bar \gamma\llcorner_{[t_1,t_2]}$,
\item $\gamma^n \rightarrow \bar\gamma$ uniformly in $[t_1,t_2]$,
\item $\displaystyle{\liminf_{n\rightarrow +\infty} T(\gamma^n,w^n) = \tilde t> t_1}$,
\end{enumerate}
it holds
\begin{equation}\label{E_l_segm}
\lim_{n\rightarrow +\infty}\dist(w^n, I)= 0.
\end{equation}
In particular $\bar\gamma$ is a segment in $[t_1,t_2]$ with velocity $f'(I)$.
\end{lemma}

\begin{proof}
\textbf{Claim 1}. Let $(\gamma^n,w^n)$ be a sequence as in the statement and $\bar w$ be a cluster point of the sequence $w^n$. Let $\eta$ be an entropy
of this type: 
\begin{equation*}
\eta_k^+(u)=(u-k)^+ \mbox{ with }k>\bar w \quad \mbox{or}\quad \eta_k^-(u)=(u-k)^- \mbox{ with }k<\bar w
\end{equation*}
and denote the relative flux by $q$. 

Then the flux from the right side 
of $\eta$ across $\bar\gamma$ in $[t_1,\tilde t]$ is zero: i.e. for almost every $t\in [t_1,\tilde t]$,
\begin{equation*}
-\dot{\bar\gamma}\langle \nu^+_t,\eta\rangle + \langle \nu^+_t,q\rangle = 0.
\end{equation*}

{\it Proof of Claim 1}. Roughly speaking the proof is the following: consider the amount of entropy between $\bar \gamma$ and $\gamma^n$ at time $t_1$.
The flux across both boundaries is non negative, in particular the flux across $\bar \gamma$ from the right is less than the total amount of entropy
at time $t_1$ between $\bar \gamma$ and $\gamma^n$. Since $\gamma^n$ is arbitrarily close to $\bar\gamma$ the flux must be 0.

Consider $\eta$ as in the statement of the claim and compute the balance in the region delimited by $\bar\gamma$ and 
$\gamma^n$ for $t\in [t_1,\tilde t]$: using that $\eta_t+q_x \le 0$,
\begin{equation}\label{E_bal_segm}
\int_{\gamma(\tilde t)}^{\gamma^n(\tilde t)}\langle \nu_{\tilde t, x}, \eta\rangle dx - 
\int_{\gamma(t_1)}^{\gamma^n(t_1)}\langle \nu_{t_1,x}, \eta\rangle dx + 
\int_{t_1}^{\tilde t}\langle \nu^-_{n,t}, -\dot\gamma^n(t)\eta + q\rangle dt - \int_{t_1}^{\tilde t}\langle \nu^+_t, -\dot{\bar\gamma}(t)\eta + q\rangle dt \leq 0,
\end{equation}
where $\nu^+$ denotes the right trace of $\nu$ on $\bar\gamma$ and $\nu^-_n$ denotes the left trace on $\gamma^n$.
Since $w^n\rightarrow \bar w$, $\eta$ is an admissible boundary entropy also for $(\gamma^n,w^n)$ for $n$ sufficiently large, so that the flux across $\gamma^n$ is non-negative: for $\mathcal L^1$-a.e. $t\in (t_1,\tilde t)$,
\begin{equation*}
\langle \nu^-_{n,t}, -\dot\gamma^n(t)\eta + q\rangle \geq 0.
\end{equation*}
Moreover for $\mathcal L^1$-a.e. $t\in (t_1,\tilde t)$
\begin{equation*}
\langle \nu^+_t, -\dot{\bar\gamma}(t)\eta + q\rangle \leq 0
\end{equation*}
because $\bar w$ is an admissible boundary on $\bar\gamma$. To prove the other inequality take the limit as $n\rightarrow \infty$ in \eqref{E_bal_segm}: since 
\begin{equation*}
\lim_{n\rightarrow \infty}\int_{\bar\gamma(\tilde t)}^{\gamma^n(\tilde t)}\langle \nu_{\tilde t, x}, \eta\rangle dx = 
\lim_{n\rightarrow \infty}\int_{\bar\gamma(t_1)}^{\gamma^n(t_1)}\langle \nu_{t_1,x}, \eta\rangle dx =0 
\end{equation*}
it holds 
\begin{equation*}
0 \le \lim_{n\rightarrow \infty}  \int_{t_1}^{\tilde t}\langle \nu^-_{n,t}, -\dot\gamma^n(t)\eta + q\rangle dt  \le \int_{t_1}^{\tilde t}\langle \nu^+_t, -\dot{\bar\gamma}(t)\eta + q\rangle dt
\end{equation*}
and this concludes the proof of Claim 1.

\textbf{Claim 2}.
Let $\bar w$ be a cluster point of the sequence $w^n$. Then the Young measures 
\begin{equation*}
\nu^1_{t,x}= \begin{cases}
\nu_{t,x} & \text{if }x<\bar\gamma(t), \\
\delta_{\bar u} & \text{if }x>\bar\gamma(t),
\end{cases}
\qquad
\text{and}
\qquad 
\nu^2_{t,x}= \begin{cases}
\delta_{\bar u} & \text{if }x<\bar\gamma(t), \\
\nu_{t,x} & \text{if }x>\bar\gamma(t),
\end{cases}
\end{equation*}
are mv entropy solutions of \eqref{E_cl} in $(t_1,\tilde t)\times \R$.

{\it Proof of Claim 2}. We need to verify that $\nu^1$ and $\nu^2$ are mv entropy solutions on $\bar\gamma$. More precisely we have to verify that for all convex entropy-entropy flux $(\eta,q)$, for $\mathcal L^1$-a.e. $t\in (t_1,\tilde t)$
\begin{equation}\label{E_nu1_sol}
Q^{1-}_{\eta,q}(t):=\langle \nu^{1-}_t, -\dot{\bar\gamma}(t)\eta+q\rangle \geq \langle \nu^{1+}_t, -\dot{\bar\gamma}(t)\eta+q\rangle=: Q^{1+}_{\eta, q}(t)
\end{equation}
and similarly for $\nu^2$.

By the previous step we know that for every boundary entropy-entropy flux $(\eta,q)$ with value $\bar w$, $Q^{1+}_{\eta,q}=Q^+_{\eta,q}=0$, where 
$Q^+_{\eta,q}(t)=\langle \nu^+_t,-\dot{\bar\gamma}(t)\eta +q\rangle$ is the flux for the real solution. We claim that this implies that $Q^+_{\eta,q}=Q^{1+}_{\eta,q}$ for every entropy-entropy
flux pair $(\eta,q)$. This is sufficient to conclude since \eqref{E_nu1_sol} holds for $\nu$.

Observe that 
\begin{equation*}
Q^+_{\eta_1\pm\eta_2,q_1\pm q_2}=Q^+_{\eta_1,q_1}\pm Q^+_{\eta_2,q_2}
\end{equation*} and that the family of finite sums with sign of boundary 
entropies is dense in the family of Lipschitz entropies with $\eta(\bar u)=0$.
Using the fact that if $\eta^n$ and $q^n$ converges uniformly to $\eta$ and $q$ respectively then $Q^+_{\eta^n,q^n}\rightarrow Q^+_{\eta,q}$ almost everywhere, 
by density $Q^+_{\eta,q}=Q^{1+}_{\eta,q}=0$ for every entropy-entropy flux pair with $\eta(\bar w)=q(\bar w)=0$ and the claim for $\nu^1$ easily follows.

Consider the entropy $\I - \bar w$ and the flux $f-f(\bar w)$. By the previous step it follows that $Q^+=0$ therefore $Q^{2+}=0$ by conservation. By definition $Q^{2-}=0$ therefore $\nu^2$ is a distributional solution.
Moreover it does not dissipate any of the boundary entropies on $\gamma$ and it is fairly easy to prove that a solution that does not dissipate any boundary
entropy does not dissipate any entropy. In particular $\nu^2$ is a mv entropy solution.

{\it Proof of Lemma \ref{L_segments}}.
In order to prove \eqref{E_l_segm} suppose there exists a sequence as in the statement such that $w^n$ has two cluster points $a\neq b$ and $\displaystyle{\liminf_{n\rightarrow +\infty} T(\gamma^n,w^n) =\tilde t>  t_1}$. We need to prove that $a$ and $b$ belong to the same linearly degenerate component of the flux. Let
\begin{equation*}
\bar t= \min\big\{\liminf_{n\rightarrow +\infty} T(\gamma^n,w^n), t_2\big\}.
\end{equation*}
Applying Claim 2 twice we get that
\begin{equation*}
u^1(t,x)= \begin{cases}
a & \text{if }x<\bar\gamma(t), \\
b & \text{if }x>\bar\gamma(t),
\end{cases}
\qquad
\text{and}
\qquad 
u^2(t,x)= \begin{cases}
b & \text{if }x<\bar\gamma(t), \\
a & \text{if }x>\bar\gamma(t),
\end{cases}
\end{equation*}
are both entropy solutions of \eqref{E_cl} in $[t_1,\bar t]\times \R$. This implies that $a$ and $b$ belong to the same linearly degenerate component of $f$ 
and that $\bar\gamma$ is a segment with velocity $f'(a)$.

By the completeness property there exists a sequence as in the statement with the additional assumption that 
\begin{equation*}
\liminf_{n\rightarrow +\infty} T(\gamma^n,w^n)\geq t_2.
\end{equation*}
Then repeating 
the argument above in $[t_1,\tilde t]$ we get that every limit of $u^n$ belongs to $I_a$, in particular $\bar\gamma$ is a segment with constant velocity $f'(a)$ for 
$t\in [t_1,t_2]$.
\end{proof}

We introduce the following partition of the half-plane. \label{P_page_decomp}
\begin{enumerate}
\item The set $A_1$ is given by points belonging to at least two curves in $\mathcal K_\gamma$:
\begin{equation*}
A_1=\Big\{( t, x) \in \R^+\times \R: \ \exists \gamma\neq\gamma' \in\mathcal K_\gamma \big(\gamma(t)=\gamma'(t)= x \big)\Big\}.
\end{equation*}
For every $(\bar t,\bar x)\in \R^+\times \R\setminus A_1$ let $\bar \gamma= \gamma_{\bar y}$ be the unique curve in $\mathcal K_\gamma$ such that $\bar\gamma(\bar t)=\bar x$.
\item The open set $B$ is given by
\begin{equation*}
B=\Big\{(\bar t,\bar x): \exists \tilde t<\bar t, \ y_1<\bar y< y_2 \, \big(\gamma_{y_1}(\tilde t)=\gamma_{y_2}(\tilde t) \ \text{and}\ \gamma_{y_1}(\bar t) < \bar x <\gamma_{y_2}(\bar t) \big)\Big\}.
\end{equation*}
\item The set $C$ is given by the points $(\bar t,\bar x) \in\R^+\times \R$ such that
\begin{equation}\label{E_c}
\forall t<\bar t, \ \forall y>\bar y \ \left(\gamma_y(t)>\gamma_{\bar y}(t)\right) \quad \mbox{and} \quad
\forall t<\bar t, \ \forall y<\bar y \ \left(\gamma_y(t)<\gamma_{\bar y}(t)\right).
\end{equation}
By Lemma \ref{L_segments} the set $C$ is obtained as the union of segments starting from 0.
\item Let $A_2$ be the complement: $A_2=\R^+\times \R\setminus (A_1\cup B \cup C)$. 
\end{enumerate}
Setting $A=A_1\cup A_2$ we have $\R^+\times\R=A\cup B\cup C$.

We will need a further distinction: 
let $A_2=A_2'\cup A_2''$, where $A_2'$ is the set of points $(t,x)$ for which only one condition in \eqref{E_c} holds and $A_2''$ is the set of points $(t,x)$ such that 
\begin{enumerate}
\item there exists a unique curve $\gamma\in \K_\gamma$ such that $\gamma(t)=x$;
\item there exists $t^-<t$ and $\gamma^-$ such that $\gamma^-(t^-)=\gamma(t^-)$ and $\gamma^-(t)<\gamma(t)$;
\item there exists $t^+<t$ and $\gamma^+$ such that $\gamma^+(t^+)=\gamma(t^+)$ and $\gamma^+(t)>\gamma(t)$;
\item there are no $\tilde \gamma^-,\tilde \gamma^+\in \K_\gamma$ and $\tilde t<t$ such that $\tilde\gamma^-(\tilde t)=\tilde\gamma^+(\tilde t)$ and $\tilde\gamma^-( t)<\gamma(t)<\tilde\gamma^+(t)$.
\end{enumerate}
See Figure \ref{F_partition}.

\begin{figure}
\centering
\def\svgwidth{\columnwidth}
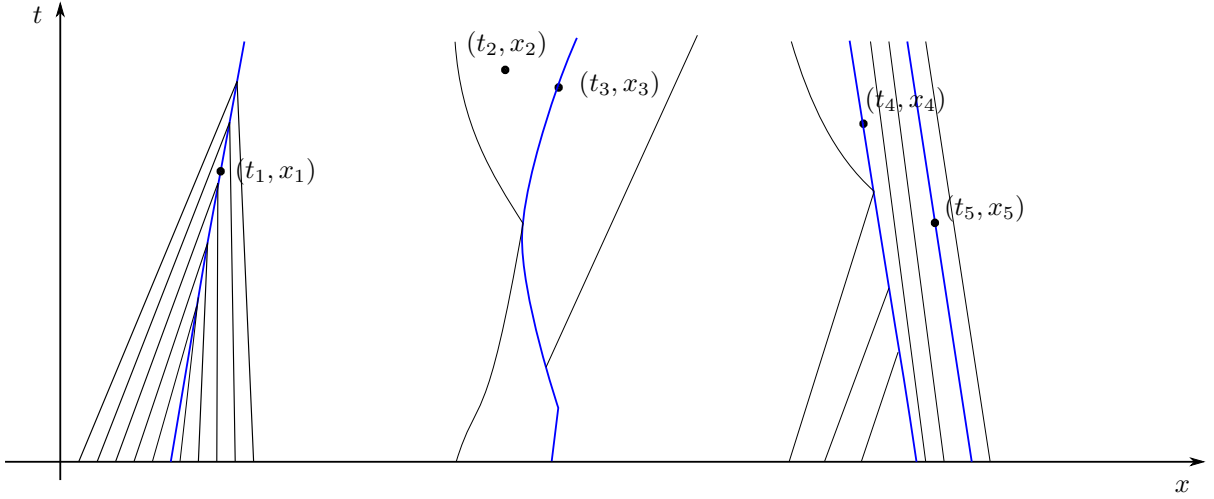
\caption{Example of points of the partition: $(t_1,x_1)\in A_1$, $(t_2,x_2)\in B$, $(t_3,x_3)\in A_2''$, $(t_4,x_4)\in A_2'$ and $(t_5,x_5)\in C$.}\label{F_partition}
\end{figure}

The candidate jump set $J$ is the set of points $(t,x)\in \R^+\times \R$ such that one of the following possibilities happens:
\begin{enumerate}
\item $(t,x)\in A$;
\item $(t,x)\in B$ and the solution is not continuous in $(t,x)$.
\item $(t,x)\in C$ such that the left and right limits obtained by Lemma \ref{L_segments} belong to different linearly degenerate components 
$I^-,I^+\in \mathcal L_f$ of the flux $f$.
\end{enumerate}

\begin{lemma}\label{L_rect}
There exists a countable subset $N\subset \mathcal K_\gamma$ such that
\begin{equation*}
J\subset \bigcup_{\gamma\in N} \graph(\gamma).
\end{equation*}
\end{lemma}
\begin{proof}
The set $A_1$ is covered by countably many curves thanks to monotonicity:
\begin{equation*}
A_1\subset \bigcup_{y\in \Q}\graph(\gamma_y).
\end{equation*}

Let $(\bar t,\bar x)\in A_2'$ such that only the first condition in \eqref{E_c} holds:  then there exists a left neighborhood $U_{\bar y}$ 
of $\bar y$ in $\R$ and a neighborhood $U_{\bar t}$ of $\bar t$ such that for every 
$(t,y)\in U_{\bar t}\times U_{\bar y}$, the first condition in \eqref{E_c} is not satisfied. Since $\R^+\times \R$ is separable it has at most countably many
disjoint open subsets and this proves the claim for $A_2'$. \\
For every $(\bar t,\bar x)\in A_2''$ there exists a neighborhood in $\R^+\times \R$ such that the points on $\gamma$ are the only points not belonging to $B$ and this concludes the proof of the statement for $A_2$, again by separability. 

The result for $J\cap B$ follows from Lemma \ref{L_BV_case}.

Write 
\begin{equation*}
C=\bigcup_{n}C_n,
\end{equation*}
where $C_n$ is the subset of $C$ such that \eqref{E_c} holds with $\bar t \ge 2^{-n}$:
clearly
\begin{equation*}
\overline{C}_n\subset C_n \cup A_2' \cup A_1.
\end{equation*}
Since every point in $\overline{C}_n$ has both limits left and right $I^-$ and $I^+$ respectively from Lemma \ref{L_segments}, they can be different at most on countably many segments of $\overline{C}_n$  because for every $m\in \N$ the points such that $\dist (I^-,I^+)>1/ m$ is discrete.
\end{proof}

The following proposition is a sort of continuity outside $J$.

\begin{proposition}\label{P_struct_Jc}
For every $(t,x)\in \R^+\times \R \setminus J$ there exists $I\in \mathscr L_f$ such that 
\begin{equation*}
\forall \e>0 \,\, \exists r >0  \,\big(U_{t,x}(r)\subset I + (-\e,\e)\big).
\end{equation*}
\end{proposition}
\begin{proof}
If $(t,x)\in B\setminus J$ the claim follows from Lemma \ref{L_BV_case}. If $(t,x)\in C\setminus J$ the claim 
follows from the definition of $J$.
\end{proof}

Similarly, the next lemma corresponds to an extension of left/right continuity at a fixed time $\bar t$.
\begin{lemma}\label{L_trace_max}
For every $(\bar t,\bar x)\in \R^+\times\R$ there exist $I^+,I^-\in \mathscr L_f$ such that 
\begin{equation*}
\forall \e>0 \,\exists r >0  \,\big(U^\pm_{\bar t,\bar x}(r)\subset I^\pm + (-\e,\e)\big).
\end{equation*}
\end{lemma}
\begin{proof}
One of the following cases occurs:
\begin{enumerate}
\item for all $\gamma>\gamma^+_{\bar t,\bar x}$ and $t< \bar t$, it holds $\gamma(t)>\gamma^+_{\bar t,\bar x}(t)$;
\item there exists $\gamma>\gamma^+_{\bar t,\bar x}$ and $t< \bar t$ such that 
$\gamma(t)=\gamma^+_{\bar t,\bar x}(t)$.
\end{enumerate}
Case (1): the claim immediately follows from Lemma \ref{L_segments}.
\newline
Case (2): in this case $\gamma^+_{\bar t.\bar x}$ is the left boundary of a Riemann problem with two boundaries. 
Consider a monotone decreasing sequence $\gamma_n \rightarrow \gamma^+_{\bar t,\bar x}$ and values $w_n$ such that
\begin{equation*}
\liminf_{n\rightarrow \infty}T(\gamma_n, w_n) \ge \bar t.
\end{equation*}
By Corollary \ref{C_uniq_K} the sequence $w_n$ is monotone and this implies the claim.
%
%
\end{proof}

Finally a result similar to the $L^1$ blow-up of $\BV$ functions.

\begin{proposition}\label{P_struct_J}
For every $\gamma \in \K_\gamma$, for $\mathcal L^1$-a.e. $t>0$, there exist $I^+,I^-\in \mathscr L_f$ such that 
\begin{equation*}
\forall \delta>0 \, \, \forall \e>0 \, \exists r>0 \, \big( U^{\delta \pm}_{t,\gamma}(r)\subset I^\pm+(-\e,\e)\big).
\end{equation*}
\end{proposition}
Recall that $U^{\delta \pm}_{t,\gamma}(r)$ is defined if $\dot\gamma(t)$ exists.
\begin{proof}
Fixed $\bar \gamma \in \K_\gamma$, we distinguish two cases as in the proof of the previous lemma: let $\R^+= T_m\cup T_s$
where 
\begin{enumerate}[(a)]
\item $\bar t \in T_m$ if there exists 
$\gamma>\gamma^+_{\bar t,\bar\gamma(\bar t)}$ and $t< \bar t$ such that 
$\gamma(t)=\gamma^+_{\bar t,\bar \gamma(\bar t)}(t)$;
\item $\bar t\in T_s$ if for all $\gamma>\gamma^+_{\bar t,\bar \gamma(\bar t)}$ and $t< \bar t$, it holds 
$\gamma(t)>\gamma^+_{\bar t,\bar\gamma(\bar t)}(t)$.
\end{enumerate}

{\it Case (a)}.
Since for every $\bar t\in T_m$ there exists $\displaystyle{\lim_{x\rightarrow \bar\gamma(\bar t)^+} u(\bar t,x)}$ as in case (2) above, by a standard application of 
Egorov theorem for $\mathcal L^1$-a.e. 
$\bar t\in T_m$ there exists the trace $u^+$ in the following sense: 
\begin{equation*}
\lim_{r\rightarrow 0}\frac{1}{r^2}\int_{B^+_{(\bar t,\bar\gamma(\bar t))}(r)}\int_\R|w-u^+(\bar t)|d\nu_{t,x}(w) dxdt = 0,
\end{equation*}
where $ B^+_{(\bar t,\bar \gamma(\bar t))}(r)= B_{(\bar t,\bar \gamma(\bar t))}(r)\cap \{(t,x):x>\bar \gamma(t)\}$. 
In particular the blow-up at $(\bar t,\bar\gamma(\bar t))$ is constant $u^+(\bar t)$ on the right side of the straight line $\{x=\dot{\bar\gamma}(\bar t)t\}$. \\
Assume that by contradiction the statement of the proposition is false in $\bar t\in T_m$ as above.
Then there exist $\delta,\e>0$ and a subsequence of rescaled solutions with an admissible boundary in 
\begin{equation*}
\bigg\{(t,x): |t|\le \frac{1}{\delta}, x=\dot{\bar\gamma}(\bar t)t+1\bigg\}
\end{equation*}
with value in $\R\setminus ( I_{u^+(\bar t)} + (-\e,\e))$.
Therefore the blow-up has an admissible boundary in $\{x>\dot{\bar\gamma}(\bar t)t\}$ with value not belonging to the linearly degenerate component $I_{u^+(\bar t)}\in \mathcal L_f$.
This contradicts Lemma \ref{L_bd_const}.

{\it Case (b)}.
By Lemma \ref{L_segments} for every $\bar t\in T_s$ the maximal characteristic $\gamma_{\bar t,\bar \gamma(\bar t)}^+$ is a
segment in $[0,\bar t]$ belonging to $C\cup A_2'$ and there exists $I^+(\bar t)\in \mathcal L_f$ such that the admissible boundary values from the right of $\gamma_{\bar t,\bar \gamma(\bar t)}^+$ converge to $I^+(\bar t)$. 
We fix an arbitrary $\e>0$ and we prove the statement for $\bar t\in T_s(\e):= T_s\cap(2\e,+\infty)$. \\
For every $\bar t\in T_s(\e)$, denote by $y(\bar t):=\gamma_{\bar t,\bar \gamma(\bar t)}^+(\e)$ and let $\gamma_{y(\bar t)}=\gamma_{\bar t,\bar \gamma(\bar t)}^+$.
In the proof of Lemma \ref{L_rect} we observed that there exists an at most countable set $N=\{y^n\}_{n\in \N}\subset y(T_s(\e))$  such that for every 
$y\in y(T_s(\e))\setminus N$, $\gamma_y\in C$ and $I^-=I^+$.
It is easy to prove that the function $y(t)$ is monotone, in particular it is continuous except an at most a countable subset of $T_s(\e)$.
Therefore we can write
\begin{equation*}
T_s(\e)= E\cup\bigcup_{n=0}^{+\infty} T_s^n(\e),
\end{equation*}
where 
\begin{enumerate}
\item $\mathcal L^1(E)=0$;
\item $t\in  T_s^0(\e)$ if and only if $t\in T_s(\e)$, $t$ is a differentiability point of $\gamma$,  $t$ is a continuity point of $y$ and $y(t)\notin N$;
\item for every $n>0$, $t\in  T_s^n(\e)$ if and only if $t\in T_s(\e)$, $t$ is a differentiability point of $\gamma$,  $t$ is a continuity point of $y$ and $y(t)=y^n$.
\end{enumerate}

We prove the statement for points of Lebesgue density one of $T_s^n$ for every $n\ge 0$. If $n>0$ it immediately follows from Lemma \ref{L_trace_max}, being the Lebesgue points of $T^n_s$ times where $\bar \gamma$ is tangent to $\gamma_{y^n}$. \\
It remains to consider the case $n=0$.
Let $R$ be the region 
\begin{equation*}
R=\bigcup_{\bar t\in T_s(\e)}\Big\{ (t,x)\in [0,\bar t]\times \R \, : x>\gamma^+_{\bar t,\bar \gamma(\bar t)}(t)\Big\}.
\end{equation*}
By definition of $T_s^0(\e)$ it follows that for every sequence $R\ni (t_n,x_n)\rightarrow (\bar t,\bar \gamma(\bar t))$ with $\bar t\in T_s^0(\e)$ and every $\gamma_n\in \K_\gamma$ such that $\gamma_n(t_n)=x_n$ it holds $\gamma_n\rightarrow \gamma^+_{\bar t,\bar\gamma(\bar t)}$. In particular, since the limits of admissible boundaries are admissible boundaries it suffices to verify that for every $\bar t\in T_s^0(\e)$ of density one
\begin{equation*}
\forall \delta>0\, \forall \e>0 \, \exists r>0 \, \big( B^{\delta +}_{\bar t, \bar \gamma}(r) \subset R\big).
\end{equation*}
By finite speed of propagation it follows from the fact that $\bar t$ has density one in $T_s^0(\e)$.
\end{proof}

The next result in this section describes the structure of the solution $\nu$ that follows from the corresponding structure of the complete family of boundaries.

\begin{corollary}\label{C_struct_Jc}
Let $\nu$ be a mv entropy solution of \eqref{E_cl} with a complete family of boundaries. Then there exists a representative of $\nu$ such that
\begin{enumerate}
\item $\langle \nu, f' \rangle$ is continuous in $\R^+\times \R \setminus J$;
\item for $\mathcal H^1$-a.e. $(t,x)\in J$, there exists $\lambda^-,\lambda^+\in \R$ and $\gamma\in \K_\gamma$ such that $\gamma(t)=x$ and for every $\delta >0$
\begin{equation*}
\lim_{r\rightarrow 0}\|\langle \nu, f' \rangle-\lambda^-\|_{L^\infty(B^{\delta-}_{t,\gamma})}=0,\qquad
\lim_{r\rightarrow 0}\|\langle \nu, f' \rangle-\lambda^+\|_{L^\infty(B^{\delta+}_{t,\gamma})}=0;
\end{equation*}
\item for every $(\bar t,\bar x)\in \R^+\times\R$ there exist left and right limits
\begin{equation*}
\lambda^-=\lim_{x\rightarrow\bar x^-}\langle \nu_{\bar t, x},f'\rangle, \qquad
\lambda^+=\lim_{x\rightarrow\bar x^+}\langle \nu_{\bar t, x},f'\rangle.
\end{equation*}
\end{enumerate}
If $f$ is weakly genuine nonlinear then $\nu=\delta_u$ is a Dirac solution and the same regularity can be deduced for $u$.
\end{corollary}
The proof is just the observation that $f'$ is constant on $I\in \mathcal L_f$ plus the fact that weak genuine nonlinearity implies that each $I\in \mathcal L_f$
is a singleton.

\begin{remark}\label{R_compactness}
Let $\nu$ be a mv solution for which there exists a complete family of boundaries. Then for almost every $(t,x)\in \R^+\times \R$, 
\begin{equation*}
\supp \,\nu_{t,x}\subset I
\end{equation*}
for some $I\in \mathscr L_f$.
\end{remark}
Suppose additionally that $f$ is weakly genuine nonlinear and $u^n\rightarrow \nu$ as Young measures where $u^n$ are entropy solutions of
\eqref{E_cl}. Then Remark \ref{R_compactness} implies that $\nu_{t,x}= \delta_{u(t,x)}$ for an $L^\infty$ entropy solution $u$ of \eqref{E_cl} and 
$u^n\rightarrow u$ strongly in $L^1(\R^+\times \R)$.

\begin{remark}\label{R_traces}
Consider a curve $\gamma \in \K_\gamma$. In Section \ref{S_prel} a notion of left and right trace has been defined for $\nu$ on $\gamma$ . 
The results in this section allow to compute the speed of $\gamma$ and the dissipation $\mu$ on $\gamma$ for every entropy $\eta$.

By Proposition \ref{P_struct_J} it follows that for $\mathcal L^1$-almost every $t>0$ there exist $I^+(t)$ and $I^-(t)$ in $\mathcal L_f$ such that $\supp \nu^\pm_t\subset I^\pm(t)$. 
If $I^+(t)\ne I^-(t)$ then the Rankine-Hugoniot condition implies that 
\begin{equation*}
\dot\gamma(t)=\frac{\langle \nu^+_t,f\rangle - \langle \nu^-_t,f\rangle}{\langle \nu^+_t,\Id\rangle - \langle \nu^-_t,\Id\rangle}.
\end{equation*}
Observe that the denominator is non zero since $I^-(t)$ and $I^+(t)$ are disjoint. Moreover for every entropy-entropy flux pair $(\eta,q)$ the dissipation along 
$\gamma$ is given by
\begin{equation}\label{E_diss_gamma}
\mu\llcorner \Graph(\gamma)= \Big(\langle \nu^+,q\rangle - \langle \nu^-,q\rangle -\dot\gamma(t)\big( \langle \nu^+,\eta\rangle - \langle \nu^-,\eta\rangle\big)\Big)
\frac{1}{\sqrt{1+\dot\gamma(t)^2}}\mathcal H^1\llcorner \Graph(\gamma).
\end{equation}

If $I^+(t)=I^-(t)$ then by Point (5) in Definition \ref{D_bd_fam} it holds $\dot\gamma(t)=f'(I^+(t))$ moreover since $q-\dot\gamma \eta$ is constant on $I^+(t)$ by \eqref{E_diss_gamma} it follows $\mu(\Graph(\gamma))=0$.
\end{remark}

\begin{remark}\label{R_V_Riem}
If the boundaries of the Riemann problem with two boundaries belong to a complete family of boundaries, we can refine Proposition \ref{P_V_Riem}: in particular using the same notation
we can prove that $\Omega^m=\Omega$. Roughly speaking this means that no constant region can appear.

By properties (3) and (5) in Proposition \ref{P_V_Riem} it suffices to prove that $\mathcal L^1$-a.e. in $(0,T)$, it holds 
$\dot\gamma_1\ge \lambda^-$ and $\dot\gamma_2 \le \lambda^+$. \\
Denote by $T^-\subset (0,T)$ the set of points where $\dot\gamma_1(t)< \lambda^-$. In particular for every $t\in T^-$, $\dot\gamma_1(t)<f'(a)$.
By Points (2) and (5), for $\mathcal L^1-$a.e. $t\in T^-$, the right trace $u^+(t)=a$.
Therefore for every entropy-entropy flux pair $(\eta,q)$ the dissipation on $\gamma_1$ for $t\in T^-$ is
\begin{equation*}
\left(-\dot\gamma_1(t)(\eta(a)-\langle \nu^-_t,\eta\rangle)+q(a)-\langle \nu^-_t,q\rangle\right)\frac{1}{\sqrt{1+\dot\gamma_1(t)^2}}
\mathcal H^1\llcorner \graph(\gamma_1\llcorner T^-).
\end{equation*}
We will obtain that $\mathcal L^1(T^-)=0$ by checking that $-\dot\gamma_1(t)(\eta(a)-\langle \nu^-_t,\eta\rangle)+q(a)-\langle \nu^-_t,q\rangle \le 0$ for every 
entropy-entropy flux pair $(\eta,q)$. \\
Indeed we already know by Corollary \ref{C_struct_Jc} that for $\mathcal L^1$-a.e. $t\in T^-$ there exists $I_t\in \mathcal L_f$ such that $\supp\nu^-_t \subset I_t$.
First we observe that $a\in I_t$ for $\mathcal L^1$-a.e. $t\in T^-$. Indeed if $a<w$ for every $w\in I_t$, check the entropy inequality for $\eta_k^-$ with 
$k\in (a,\inf I_t)$:
\begin{equation*}
\begin{split}
0 \ge &~ -\dot\gamma_1(t)(\eta^-_k(a)-\langle \nu^-_t,\eta_k^-\rangle)+q_k^-(a)-\langle \nu^-_t,q^-_k\rangle \\
= &~  -\dot\gamma_1(t)(k-a)+f(k)-f(a) \\
= &~ (f'(a)-\dot\gamma_1(t))(k-a) + o(|k-a|).
\end{split}
\end{equation*}
Since $\dot\gamma_1(t)<f'(a)$, the inequality above cannot be satisfied for $k$ in a right neighborhood of $a$.
Similarly the case $a>w$ for $w\in I_t$ is excluded.
Then the conclusion easily follows: by property (5) in Definition \ref{D_bd_fam}, we have that $\dot\gamma^1(t)=f'(a)$ for $\mathcal L^1$-a.e. $t\in T^-$, 
therefore $\mathcal L^1(T^-)=0$.
 \end{remark}

\section{Concentration}\label{S_conc}
In this section we study the structure of the dissipation measure $\mu=\langle \nu,\eta\rangle_t+\langle \nu, q\rangle_x$, where $(\eta,q)$ is an 
entropy-entropy flux pair and $\nu$ is a mv entropy solution with a complete family of boundaries.

Consider the decomposition of $\R^+\times \R$ introduced in Section \ref{S_struct}: the dissipation on $J$ can be computed by means of the traces
given in Proposition \ref{P_traces}: see Remark \ref{R_traces}. Moreover $\mu(B\setminus J)=0$ by Volpert chain rule for functions of bounded variation.
Here we analyze $\mu\llcorner (C\setminus J)$.

Let $\e>0$, $T>2\e$ and consider the set $S_T=\{x\in \R: (T,x)\in C\setminus J\}$. 
By Lemma \ref{L_segments}, for all $x\in S_T$ the unique curve $\gamma\in \K_\gamma$ such that $\gamma(T)=x$ has constant velocity $f'(I)$
in $[0,T]$ where $I$ is the unique element of $\mathcal L_f$ such that $K(T,x)\subset I$. 
Denote this set of curves by $\K_\gamma(T)$ and parametrize it by the position $y=\gamma_y(\e)$ of the curves at time $\e$.
Moreover set 
\begin{equation*}
Y_T:=\Big\{y\in \R: \gamma_y\in \K_\gamma(T)\Big\}
\end{equation*} 
and let $I(y)$ be the corresponding element of $\mathcal L_f$.

\begin{lemma}\label{L_Flip}
There exists $U\in L^\infty (\R)$ such that for every $y\in Y_T$, $U(y)\in K(\e,y)$ and
\begin{enumerate}
\item for every entropy-entropy flux pair $(\eta,q)$ the function
\begin{equation*}
Q(y)=q(U(y))-f'(U(y))\eta(U(y))
\end{equation*}
has locally bounded variation;
\item in the particular case with $(\eta,q)=(\I,f)$ the function 
\begin{equation*}
F(y)=f(U(y))-f'(U(y))U(y)
\end{equation*}
has no Cantor part.
\end{enumerate}
\end{lemma}
\begin{proof}
Since the segments do not cross in $(0,T)$, by monotonicity for every $y_1,y_2\in Y_T$
\begin{equation*}
\big|f'(I(y_2))-f'(I(y_1))\big|\le \frac{1}{\e}|y_2-y_1|.
\end{equation*}
Then consider the domain 
\begin{equation*}
\mathcal{C}_{y_1,y_2}(T) = \Big\{ (t,x): t\in (0,T), \gamma_{y_1}(t)<x<\gamma_{y_2}(t)\Big\}.
\end{equation*}
Proposition \ref{P_traces} allows the application of the divergence theorem on $\mathcal C_{y_1,y_2}(T)$ so we get
\begin{equation}\label{E_bal_cyl}
\int_{\gamma_{y_1}(T)}^{\gamma_{y_2}(T)}\langle \nu^-_{T,x}, \eta \rangle dx - \int_{\gamma_{y_1}(0)}^{\gamma_{y_2}(0)}\langle \nu^+_{0,x}, \eta \rangle dx + T(Q(y_2)-Q(y_1)) = \mu (\mathcal C_{y_1,y_2}(T)),
\end{equation}
where $Q(y)=q(I(y))-f'(I(y))\eta(I(y))$ is well-defined being constant on each $I\in \mathcal L_f$.

Since $\{\gamma_y\}_{y\in Y_T}$ are segments in $[0,T]$ which do not cross in $(0,T)$ thanks to the monotonicity property, 
for every $y_1<y_2$ in $Y_T$
\begin{equation}\label{E_no_cross}
0\le \gamma_{y_2}(T)-\gamma_{y_1}(T) < \frac{T}{\e} \qquad \mbox{and}\qquad \gamma_{y_2}(0)-\gamma_{y_1}(0)<\frac{T}{T-\e}<\frac{T}{\e}.
\end{equation}
Therefore from \eqref{E_bal_cyl} it follows that there exists a constant $C$ depending on $f, \eta, \|\nu\|_{\infty}$ such that 
\begin{equation*}
\big|Q(y_2)-Q(y_1)\big| \le \frac{C}{\e}(y_2-y_1) + \frac{1}{T}|\mu| (\mathcal C_{y_1,y_2}(T)).
\end{equation*}
It follows that for every $L>0$
\begin{equation*}
\sup_{y_1< \ldots <y_n \in  Y_T\cap [-L,L]}\sum_{i=1}^{n-1}|Q(y_{i+1})-Q(y_i)| \le \frac{2CL}{\e}+ \frac{1}{T}|\mu|\Big((0,T)\times (-L-CT,L+CT)\Big) <+\infty
\end{equation*}
and that $F\llcorner Y_T$ is Lipschitz for every section $U$ of $K(\e,\cdot)$. Then it is easy to show that $U$ can be extended maintaining the required
properties e.g. taking $U(y)\in I(y)$ where
\begin{equation*}
I(y)=\lim_{y'\searrow \inf\{Y_T\cap [y,+\infty)\}}I(y').
\end{equation*}
The limit exists by Lemma \ref{L_trace_max}.
\end{proof}

The following geometric lemma has a quite standard proof. We give it for completeness.
\begin{lemma}\label{L_geom}
Let $\alpha:[-M,M]\rightarrow \R^2$ be a smooth curve and assume that there exists a constant $C>0$ such that
\begin{equation}\label{E_Lip}
|\dot\alpha^2|\leq C |\dot\alpha^1|.
\end{equation}
Let $I\subset \R$ be an interval and $\gamma=(\gamma^1,\gamma^2)=\alpha\circ \varphi$ for some Borel $\varphi: I\rightarrow [-M,M]$; suppose that 
$\gamma$ has bounded variation and $\gamma^1\in \SBV(I)$. Then $\gamma^2 \in \SBV(I)$ and there exists $c:I\rightarrow \R$ 
such that for $\mathcal L^1$-a.e. $y\in I$ 
\begin{equation*}
D_y\gamma(y)= c(y)D_w\alpha(\varphi(y)).
\end{equation*}
\end{lemma}
\begin{proof}
Let $\tilde\gamma=(\tilde \gamma^1,\tilde\gamma^2):[0,\TV(\gamma)]\rightarrow \R^2$ be the unique curve 1-Lipschitz curve such that there exists a monotone increasing function
$\psi:I\rightarrow [0,\TV(\gamma)]$ satisfying $\tilde\gamma \circ \psi=\gamma$.
For $\mathcal L^1$-a.e. $s\in \psi(I)$ there exists a unique $y \in I$ such that 
\begin{equation}\label{E_tang}
\psi(y)=s,\quad |D_w\alpha(\varphi(y))|\ne 0\quad \mbox{and} \quad D_s\tilde \gamma(s)=\frac{D_w\alpha(\varphi(y))}{|D_w\alpha(\varphi(y))|},
\end{equation}
because $\psi$ is monotone, $\tilde\gamma(\psi(I))\subset \Graph\, \alpha$ and $\mathcal H^1(\alpha(\{D_w\alpha=0\}))=0$.
Since $|\dot\alpha^2|\leq C |\dot\alpha^1|$, this implies that 
\begin{equation*}
|D_s\tilde\gamma^2|\llcorner \psi(I) \le C |D_s\tilde \gamma^1|\llcorner \psi(I).
\end{equation*}
Therefore $|D^c_y\gamma^2|\le C|D_y\gamma^1|$, in particular $\gamma^2\in \SBV(I)$. 
Moreover it follows from \eqref{E_tang} that there exists $c:I\rightarrow \R$ such that
 for $\mathcal L^1$-a.e. $y\in I$ 
\begin{equation*}
D_y\gamma(y)= c(y)D_w\alpha(\varphi(y)). \qedhere
\end{equation*}
\end{proof}
In our context let $(\eta,q)$ be an entropy-entropy flux pair, $\alpha:[-M,M]\rightarrow \R^2$ defined by
\begin{equation*}
\alpha(w) = \left(
\begin{array}{c}
f(w)-f'(w)w  \\
q(w)-f'(w)\eta(w)
\end{array}
\right)
\end{equation*}
and $\varphi=U$ introduced in Lemma \ref{L_Flip}. The hypothesis on $\gamma^1$ are satisfied by Lemma \ref{L_Flip}, moreover
\begin{equation*}
\dot\alpha(w)=  \left(
\begin{array}{c}
-f''(w)w  \\
-f''(w)\eta(w)
\end{array}
\right)
\end{equation*}
therefore \eqref{E_Lip} is satisfied for every $(\eta,q)$ with $\eta(0)=0$. This is not a restrictive condition since it is sufficient to subtract a constant.

Denote by 
\begin{equation*}
\widetilde{C}(T)=\bigcup_{y \in Y_T}\Graph(\gamma_y\llcorner (0,T))
\end{equation*} 
and let $P:\widetilde{C}(T)\rightarrow \R$ be the map which assigns to each $(t,x)\in \widetilde{C}(T)$ the parameter $y\in Y_T$ such that $\gamma_y(t)=x$.

The following corollary is a first result toward the concentration of entropy dissipation for mv entropy solutions with a complete family of boundaries.
The analysis of the endpoints of segments in $\widetilde{C}(T)$ will be done in Lemma \ref{L_endpoints}. 
\begin{corollary}\label{C_Cantor_part}
For every entropy-entropy flux pair $(\eta,q)$ with dissipation measure $\mu$ the Cantor part of $P_\sharp (\mu\llcorner_{\widetilde{C}(T)})$  vanishes.
\end{corollary}
\begin{proof}
By \eqref{E_bal_cyl} and \eqref{E_no_cross} it follows that there exists a constant $C$ such that 
\begin{equation*}
|P_\sharp (\mu\llcorner_{\widetilde{C}(T)})|\leq T|D_yQ| + C\frac{T}{\e}\mathcal L^1
\end{equation*}
and Lemma \ref{L_geom} implies that $Q$ belongs to $SBV_\loc$, therefore $P_\sharp (\mu\llcorner_{\widetilde{C}(T)})$ has no Cantor part.
\end{proof}

Denote by
\begin{equation*}
L_T=\Big\{y\in Y_T : \exists w \in I \mbox{ for some nontrivial }I\in \mathcal L_f \mbox{ such that } (\gamma_y,w)\in \K\Big\}.
\end{equation*}
Observe that, being the isolated points of a subset of $\R$ at most countably many, we can find a set $\tilde L_T$ such that 
$L_T\setminus \tilde L_T$ is at most countable and for every $y\in\tilde L_T$ there exist a sequence $y_n\rightarrow y$, 
an interval $I\in \mathcal L_f$ and $u_n\in I$ such that $(\gamma_{y_n},u_n)\in \K$. \\
In particular for $\mathcal L^1$-a.e. $y\in L_T$
\begin{equation}\label{E_const_vel}
\mathtt d_y \dot\gamma_y=0,
\end{equation}
where $\mathtt d_y$ denotes the limit of incremental ratios with values in $Y_T$.

In the following lemma we prove that the average $\langle \nu,\I\rangle$ is a constant $\bar u(y)$ on $\gamma_y$ in $(0,T)$ for every $y\in Y_T$.
\begin{lemma}\label{L_aver_const}
For $\mathcal L^1$-a.e. $y\in Y_T$ there exists $\bar u(y)\in K(\e,y)$ such that for $\mathcal L^1$-a.e. $t\in (0,T)$
\begin{equation*}
\langle\nu_{t,\gamma_y(t)},\I\rangle = \bar u(y).
\end{equation*}
\end{lemma}
\begin{proof}
By the previous analysis we already know that for $\mathcal L^1$-a.e. $y\in Y_T$ there exists $I(y)$ such that for $\mathcal L^1$-a.e. $t\in (0,T)$
\begin{equation*}
\langle\nu_{t,\gamma_y(t)},\I\rangle \in I(y).
\end{equation*}
In particular the claim is trivial if $y\in Y_T\setminus L_T$, where $I(y)$ is a singleton.
Now consider $y\in \tilde L_T$ such that $\partial_y\dot\gamma_y=0$ and $0<t_1<t_2<T$ such that 
$\nu_1:=\nu_{t_1,\gamma_y(t_1)}$ and $\nu_2:=\nu_{t_2,\gamma_y(t_2)}$ 
are Lebesgue points of $\nu_{t_1}$ and $\nu_{t_2}$ respectively. Consider a sequence $y_n\rightarrow y$ as above: the conservation in $\mathcal C_{y_n,y}(t_1,t_2)$ gives
\begin{equation}\label{E_cons}
\begin{split}
0=&~\frac{1}{y-y_n}\left(\int_{\gamma_{y_n}(t_2)}^{\gamma_{y}(t_2)} \langle \nu_{t_2,x},\I\rangle dx -
\int_{\gamma_{y_n}(t_1)}^{\gamma_{y}(t_1)} \langle \nu_{t_1,x},\I\rangle dx + (t_2-t_1)(F(y)-F(y_n))\right) \\
= &~ \frac{1}{y-y_n}\left(\int_{\gamma_{y_n}(t_2)}^{\gamma_{y}(t_2)} \langle \nu_{t_2,x},\I\rangle dx -
\int_{\gamma_{y_n}(t_1)}^{\gamma_{y}(t_1)} \langle \nu_{t_1,x},\I\rangle dx \right)
\end{split}
\end{equation}
because $F$ is constant on $I(y)$.
Since $\gamma_y(t)=y+(t-\e)\dot\gamma_y$ and $\partial_y \dot\gamma_y=0$ by \eqref{E_const_vel},
\begin{equation*}
\lim_{n\rightarrow \infty}\frac{\gamma_y(t_2)-\gamma_{y_n}(t_2)}{y-y_n}=\lim_{n\rightarrow \infty}\frac{\gamma_y(t_1)-\gamma_{y_n}(t_1)}{y-y_n}=1,
\end{equation*}
therefore taking the limit as $n\rightarrow\infty$ in \eqref{E_cons} we get $\langle\nu_1,\I\rangle=\langle\nu_2,\I\rangle$.
\end{proof}

At this point we can obtain the chain rule corresponding to $(f(u)-f'(u)u)_y=-(f'(u))_yu$.

\begin{lemma}\label{L_chain_rule}
For $\mathcal L^1$-a.e. $y\in Y_T$ it holds
\begin{equation*}
\mathtt d_y F(y)=-\bar u (y) \mathtt d_y \dot \gamma_y.
\end{equation*}
\end{lemma}
\begin{proof}
If $y\in L_T$ the claim follows from \eqref{E_const_vel}. If $y\in Y_T\setminus L_T$ consider again the conservation \eqref{E_cons}. In this case only
the first equality holds but by Lemma \ref{L_aver_const} we can compute 
\begin{equation*}
\lim_{n\rightarrow \infty}\frac{1}{y-y_n}\left(\int_{\gamma_{y_n}(t_2)}^{\gamma_{y}(t_2)} \langle \nu_{t_2,x},\I\rangle dx -
\int_{\gamma_{y_n}(t_1)}^{\gamma_{y}(t_1)} \langle \nu_{t_1,x},\I\rangle dx\right) = \bar u(y)(t_2-t_1)\mathtt d_y\dot\gamma_y
\end{equation*}
and this completes the proof.
\end{proof}

We introduce the set $D_T$ of points $y\in Y_T$ for which $\nu_{t,\gamma_y(t)}=\delta_{\bar u(y)}$ for $\mathcal L^1$-a.e. $t\in (0,T)$. 
In particular we have that $Y_T\setminus L_T\subset D_T$.

\begin{lemma}\label{L_ac_part}
For every $\varphi\in C_c(\R^+\times\R)$
\begin{equation}\label{E_diss_form}
\begin{split}
\int_{\widetilde{C}(T)}\varphi d \mu &=~ \int_{Y_T}\int_0^T \varphi(t,\gamma_y(t))\,d\big(\partial_t\langle \nu_{t,\gamma_y(t)},\eta\rangle\big) dy \\
& =~  \int_{Y_T\setminus D_T}\int_0^T \varphi(t,\gamma_y(t))\,d\big(\partial_t\langle \nu_{t,\gamma_y(t)},\eta\rangle\big) dy. 
\end{split}
\end{equation}
In particular if $\nu$ is a Dirac entropy solution then $\mu\llcorner \widetilde{C}(T)=0$.
\end{lemma}
\begin{proof}
It suffices to prove that for $\mathcal L^2$-a.e. $0<t_1<t_2<T$,
\begin{equation*}
P_\sharp \mu\llcorner \big(\widetilde{C}(T)\cap ((t_1,t_2]\times \R)\big)= \left\langle\nu_{t_2,\gamma_y(t_2)}-\nu_{t_1,\gamma_y(t_1)},\eta\right\rangle \mathcal L^1(dy).
\end{equation*}
By Corollary \ref{C_Cantor_part} and the definition of $\widetilde{C}(T)$ we have that 
\begin{equation*}
P_\sharp \mu\llcorner  \big(\widetilde{C}(T)\cap ((t_1,t_2)\times \R)\big)\ll \mathcal L^1
\end{equation*}
so we have to check that for $\mathcal L^1$-a.e. $y\in Y_T$ the Radon-Nykodim derivative is 
$\left\langle\nu_{t_2,\gamma_y(t_2)}-\nu_{t_1,\gamma_y(t_1)},\eta\right\rangle$.

As before we distinguish the cases $y\in L_T$ and $y\in Y_T\setminus L_T$. 
Consider $y\in \tilde L_T$ which is a Lebesgue point for $\nu_{t_1,\gamma_y(t_1)}$ and $\nu_{t_2,\gamma_y(t_2)}$. 
Then, similarly to Lemma \ref{L_aver_const} for every entropy-entropy flux pair $(\eta,q)$
\begin{equation*}
\begin{split}
\frac{1}{y-y_n}\mu\big(&\mathcal C_{y_1,y_2}(T)\cap ((t_1,t_2)\times \R)\big)  = \\
= &~ \frac{1}{y-y_n}\left(\int_{\gamma_{y_n}(t_2)}^{\gamma_{y}(t_2)} \langle \nu_{t_2,x},\eta\rangle dx -
\int_{\gamma_{y_n}(t_1)}^{\gamma_{y}(t_1)} \langle \nu_{t_1,x},\eta\rangle dx + (t_2-t_1)(Q(y)-Q(y_n))\right) \\
= &~ \frac{1}{y-y_n}\left(\int_{\gamma_{y_n}(t_2)}^{\gamma_{y}(t_2)} \langle \nu_{t_2,x},\eta\rangle dx -
\int_{\gamma_{y_n}(t_1)}^{\gamma_{y}(t_1)} \langle \nu_{t_1,x},\eta\rangle dx \right),
\end{split}
\end{equation*}
because $Q$ is constant on $I(y)$
and taking the limit as $n\rightarrow \infty$
\begin{equation*}
\lim_{n\rightarrow \infty}\frac{1}{y-y_n}\mu\big(\mathcal C_{y_1,y_2}(T)\cap ((t_1,t_2)\times \R)\big)  = \left\langle\nu_{t_2,\gamma_y(t_2)}-\nu_{t_1,\gamma_y(t_1)},\eta\right\rangle.
\end{equation*}

Now we consider the case $y\in Y_T\setminus L_T$:
by Lemma \ref{L_geom} and \ref{L_chain_rule} for $\mathcal L^1$-a.e. $y\in Y_T$,
\begin{equation}\label{E_chain_rule}
\mathtt d_yQ(y)=-\eta(\bar u(y))\mathtt d_y\dot\gamma_y .
\end{equation}
Since $L_T\setminus \tilde L_T$ is at most countable it is sufficient to consider $y\in Y_T\setminus L_T$ of $\mathcal L^1$ density one so that 
$\nu_{t_1,\gamma_y(t_1)}=\delta_{u(y)}=\nu_{t_2,\gamma_y(t_2)}$ are Lebesgue points of $\nu$ and assume that \eqref{E_chain_rule} holds. \\
For every sequence $y_n\rightarrow y$ in $Y_T$ and for every $t\in (0,T)$ it holds 
\begin{equation*} 
\lim_{n\rightarrow \infty}\frac{\gamma_y(t)-\gamma_{y_n}(t)}{y-y_n}=1+ (t-\e)\mathtt d_y\dot\gamma_y
\end{equation*}
therefore the balance in $\mathcal C_{y_1,y_2}(T)\cap ((t_1,t_2)\times \R)$ gives
\begin{equation*}
\lim_{n\rightarrow\infty}\frac{1}{y-y_n}\mu\big(\mathcal C_{y_1,y_2}(T)\cap ((t_1,t_2)\times \R)\big)  = (t_2-t_1)\eta(\bar u(y))\mathtt d_y\dot\gamma_y + (t_2-t_1)\mathtt d_yQ(y) =0. \qedhere
\end{equation*}
\end{proof}

\begin{remark}\label{R_gammax}
Consider the function $P_0(y)=\gamma_y(0)$ defined on $Y_T$. Observe that $P_0$ is monotone and for $\mathcal L^1$-a.e. $y\in L_T$, it holds
$P_0'(y)=1$. In particular 
\begin{equation*}
{P_0}_\sharp \mathcal L^1 \llcorner L_T = \mathcal L^1\llcorner P_0(L_T).
\end{equation*}
Therefore we can write the formula \eqref{E_diss_form} in the following form:
\begin{equation*}
\int_{\widetilde{C}(T)}\varphi d \mu =  \int_{P_0(L_T)}\int_{(0,T)} \varphi(t,\gamma^x(t))d\left(\partial_t\langle \nu_{t,\gamma^x(t)},\eta\rangle\right) dx,
\end{equation*}
where $\gamma^x$ denote the curve in $\K$ for which $\gamma^x(0)=x$: we already observed that it is well-defined on a set whose complement 
is at most countable.
\end{remark}

In the last part of this section we study the endpoints of the segments in $C$. \\
For every $\gamma\in \K_\gamma$ let 
\begin{equation*}
T_1(\gamma)=\inf \Big\{t: \exists \gamma'\in \K_\gamma \mbox{ such that } \gamma'\ne \gamma, \gamma'(t)=\gamma(t)\Big\}.
\end{equation*}
Therefore, denoting by $\gamma^x\in \K_\gamma$ the curve starting from $x$ as in Remark \ref{R_gammax}, write
\begin{equation*}
S=C\setminus J= \bigcup_{x\in X_1}\big\{(t,\gamma^x(t)): t\in (0,T_1(\gamma^x)]\big\} \cup \bigcup_{x\in X_2}\big\{(t,\gamma^x(t)): t \in (0, T_1(\gamma^x))\big\},
\end{equation*}
where $x\in X_2$ if the infimum in the definition of $T_1(\gamma^x)$ is an actual minimum, and $x\in X_1$ when it is not a minimum. 
Let $X=X_1 \cup X_2$ and denote by 
\begin{equation*}
E:=\big\{(T_1(\gamma^x),\gamma^x(T_1(\gamma^x))): x \in X_1\big\}
\end{equation*}
the set of endpoints of segments in 
$C\setminus J$ and let $\tilde S=S\setminus E$.

Iterating the argument above on a countable dense set of $t$ in $\R^+$ we obtain that
\begin{equation*}
\int_{\tilde S}\varphi d \mu =  \int_{X}\int_{(0,T_1(\gamma^x))} \varphi(t,\gamma^x(t))\,d(\partial_t\langle \nu_{t,\gamma^x(t)},\eta\rangle) dx.
\end{equation*}

It remains to analyze the dissipation on $E$. 
In \cite{AdL_note} it is provided an example for which $\mathcal L^1(X_1)>0$.  

As in the previous argument, fix $\e>0$ and consider $E_\e=\{(t,x)\in E : t\ge 2\e\}$. Denote by 
$P: E_\e\rightarrow \R$ the map that at each $(t,x)\in E_\e$ assigns the unique $y=\gamma_y(\e)\in \R$ such that $\gamma_y(t)=x$ 
and by $Y_\e$ the image $P(E_\e)$.
Moreover we denote by $D_\e\subset Y_\e$ the set of points $y$ of density 1 for $\mathcal L^1\llcorner Y_\e$ such that $\nu_{\e,y}=\delta_{\bar u(y)}$ is a Lebesgue point
of $\nu_\e$.
We also introduce the function 
\begin{equation*}
\Phi:E_\e \rightarrow \Graph(T_1\llcorner X_1)=: G, \qquad \Phi (t,x)=(t,P(t,x)).
\end{equation*}
We observe that $\Phi$ is invertible and, since the segments $\gamma_y$ for $y\in Y_\e$ do not cross, it is fairly easy to check that $\Phi^{-1}$ is L-Lipschitz.

We say that $(t,y)\in G$ can be approximated from the right if there exists a sequence $(t_n,y_n)\in G$ converging to $(t,y)$
such that $y_n>y$ and $t_n>t+y_n-y$. Similarly we say that $(t,y)$ can be approximated from the left if  there exists a sequence 
$(t_n,y_n)\in G$ converging to $(t,x)$ such that $y_n<y$ and $t_n>t+y-y_n$.
We denote by $F$ the set of points of $G$ which can be approximated from both sides. A standard argument proves that 
$R:=G\setminus F$  is countably 1-rectifiable and being $\Phi^{-1}$ Lipschitz, $\Phi^{-1}(R)$ is also countably 1-rectifiable. See for example
\cite[Chapter~2]{AFP}.

\begin{figure}
\centering
\def\svgwidth{\columnwidth}
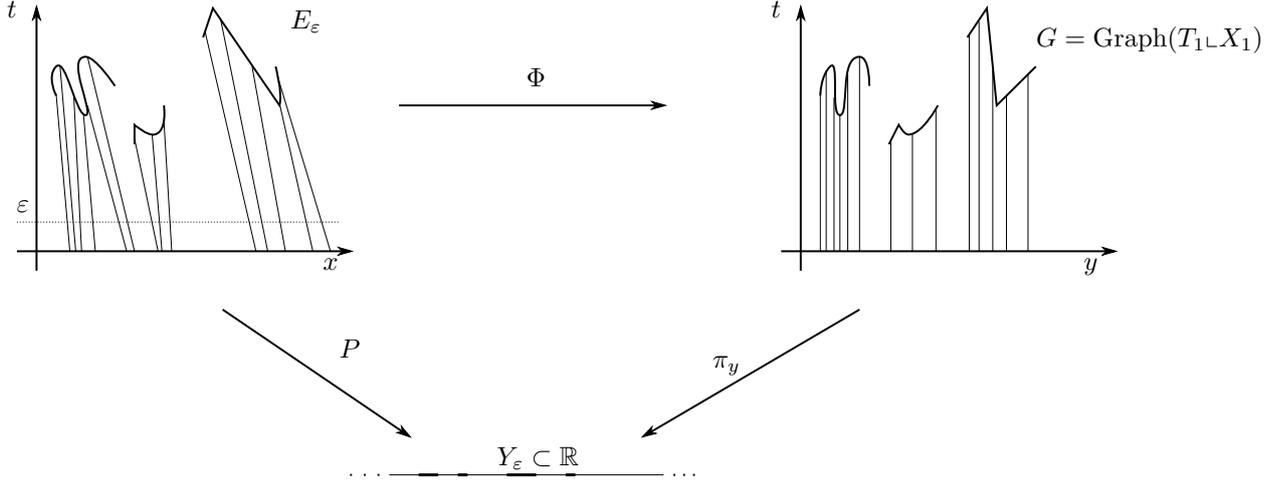
\caption{The set of endpoints in the two coordinate systems.}
\end{figure}

\begin{lemma}\label{L_endpoints}
The image measure $m:=P_\sharp (\mu \llcorner E_\e)$ is absolutely continuous with respect to $\mathcal L^1$. Moreover
\begin{equation*}
m\llcorner D_\e = 0.
\end{equation*}
\end{lemma}
\begin{proof}
We consider separately $\mu\llcorner \Phi^{-1}(F)$ and $\mu\llcorner \Phi^{-1}(R)$: in the first case we take advantage of the fact that these points can be approximated from both sides to
repeat the argument of the previous section including end-points, in the second case, being $ \Phi^{-1}(R)$ rectifiable, we can use a 
blow-up technique.

\textbf{Non-rectifiable part $F$}. 
Assume by contradiction that there exists $A\subset \pi_y(F)$ such that $m(A)>0$ and $\mathcal L^1 (A)=0$. Without loss of generality we can take $A$ compact and $T_1\llcorner A$ continuous.

We first prove that 
\begin{equation}\label{E_cover}
\forall \bar y\in A\, \forall \e>0 \, \exists \,y^-<\bar y < y^+ : |y^+-y^-|< 2\e \mbox{ and }\mathcal C_{y^-,y^+}(T_1(\bar y)+\e) \supset \Graph(T_1\llcorner (A\cap (y^-,y^+)).
\end{equation}
Define the function
\begin{equation*}
\varpi (l) = 
\begin{cases}
\displaystyle{\max \big\{T_1(y): y \in [\bar y + l , \bar y] \cap A\big\}} & l<0, \\
\displaystyle{\max \big\{T_1(y): y \in [\bar y , \bar y+l] \cap A\big\}} & l\ge 0.
\end{cases}
\end{equation*}
The function $\varpi$ is upper semicontinuous, so that  
\begin{equation*}
\bar y\in (\tilde y^-,\tilde y^+)= \varpi^{-1}([0, T_1(\bar y)+\e)).
\end{equation*}
Define 
\begin{equation*}
y^+
\begin{cases}
= \tilde y^+ & \tilde y^+ \le \e, \\
\in (\bar y,\bar y +\e] \cap T_1^{-1}([T_1(\bar y)+\e, +\infty)) & \mbox{otherwise}.
\end{cases}
\end{equation*}
The last set is nonempty by the assumption on $F$.
For $y^-$ the definition is analogue and this gives \eqref{E_cover}.

Being a fine cover, 
 for every $\delta>0$ there exists $y_i^-,y_i^+$ for $i=1,\ldots, n$ such that 
\begin{enumerate}
\item $\displaystyle{\sum_{i=1}^n y_i^+-y_i^- < \delta}$,
\item $\displaystyle{\sum_{i=1}^n |D_yQ|(y_i^-,y_i^+) < \delta}$ by Corollary \ref{C_Cantor_part},
\item $\displaystyle{\bigcup_{i=1}^n \mathcal C_{y_i^-,y_i^+}(T^i)} \supset \Graph(T_1\llcorner A)$, where $T^i=\min (T_1(\gamma_{y_i^-}), T_1(\gamma_{y_i^+}))$. 
\end{enumerate}

Computing the balance in each cylinder we get 
\begin{equation*}
\begin{split}
|\mu|(\mathcal C_{y_i^-,y_i^+}(T^i))\le &~ \int_{\gamma_{y_i^-}(0)}^{\gamma_{y_i^+}(0)} \langle \nu_{0,x}^+,\eta\rangle dx-
\int_{\gamma_{y_i^-}(T^i)}^{\gamma_{y_i^+}(T^i)}\langle \nu^-_{T^i,x},\eta\rangle dx + T^i|Q(y_i^+)-Q(y_i^-)| \\
\le &~ C(y_i^+-y_i^-)+ |D_y Q|(y_i^+-y_i^-).
\end{split}
\end{equation*}
Summing in $i$ we get $m(A)< (C+1) \delta$ and by arbitrariness of $\delta>0$ this proves that $m\llcorner \pi_y(F) \ll \mathcal L^1$.

Moreover the same covering argument allows to repeat computations in Lemma \ref{L_ac_part} and this yields that the Radon-Nikodym derivative 
of $m$ with respect to $\mathcal L^1$ vanishes in $D_\e$.

\textbf{Rectifiable part $R$}. The dissipation measure on $\Phi^{-1}(R)$ has the form $\mu\llcorner \Phi^{-1}(R)= g\mathcal H^1\llcorner \Phi^{-1}(R)$ for some 
$g\in L^\infty(\Phi^{-1}(R),\mathcal H^1)$, being the divergence of an $L^\infty$ vector field. 
We consider  a blow-up $\nu^{\infty}$ of $\nu$ at the points $z\in \Phi^{-1}(R)$ such that 
$z$ is a Lebesgue point of $g$ and the blow-up of $\Phi^{-1}(R)$ at $z$ is a straight line $R^\infty$.  
Since $z\notin J$ there exists $I(z)\in \mathcal L_f$ such that $\supp\nu^{\infty}\subset I(z)$. 

We consider two cases:
\begin{enumerate}
\item the tangent to $\Phi^{-1}(R)$ at $z$ has the same direction of $(1,f'(I(z)))$;
\item the tangent to $\Phi^{-1}(R)$ at $z$ is $(\alpha^1,\alpha^2)$ not parallel to $(1,f'(I(z)))$.
\end{enumerate}
By Remark \ref{R_traces} it follows immediately that in the first case the dissipation $\mu^\infty$ of $\nu^\infty$ on $\R^\infty$ is zero. In particular $g(z)=0$ for $\mathcal H^1$-a.e. point in $\Phi^{-1}(R)$ such that the tangent has direction $(1,f'(I(z)))$. Denote the image of this set through $\Phi$ by $R_\parallel$. 
Then it follows that
\begin{equation*}
m\llcorner \pi_y(R_\parallel)=0.
\end{equation*}
In the second case an easy computation shows that 
\begin{equation*}
m=m\llcorner \pi_y(R\setminus R_\parallel)= \frac{g(P^{-1}(y))}{\left| \alpha^2(P^{-1}(y))-\alpha^1(P^{-1}(y))f'(I_{P^{-1}(y)})\right|}\mathcal L^1(dy),
\end{equation*}
in particular it is absolutely continuous.
 
To prove that $m\llcorner D_\e=0$ we show that $g( P^{-1}(y))=0$ for $\mathcal L^1$-almost every $y\in D_\e$. 
Consider a blow-up $\nu^\infty$ of $\nu$ at a point $z\in \Phi^{-1}(R\setminus R_\parallel)$ as above with the additional requirement that $P(z)\in D_\e$.  
By the dissipation formula \ref{E_diss_form} and the definition of $D_\e$ it follows that $\nu^\infty$ is a mv entropy solution on the plane with 
$\nu^\infty=\delta_{\bar u}$ for some constant $\bar u$ on the half-plane $\alpha^2(z)t-\alpha^1(z)x<0$ where the sign of $\alpha$ has been chosen so that 
\begin{equation}\label{E_sign}
\alpha^2(z)>f'(I(z))\alpha^1(z).
\end{equation}
The dissipation on $R^\infty=\{t,f'(I(z))t\}_{t\in \R}$ is given by
\begin{equation}\label{E_diss_inf}
 \alpha^2(z) (\langle\nu^+,\eta\rangle-\eta(\bar u)) - \alpha^1(z)(\langle\nu^+,q\rangle- q(\bar u)) ,
\end{equation}
where $\nu^+$ is the trace on $\R^\infty$ of $\nu^\infty$ from the half-plane $\alpha^1(z)t+\alpha^2(z)x>0$.
Since $\supp \nu^+ \subset I(z)$, imposing that the dissipation \eqref{E_diss_inf} is nonpositive for every entropy-entropy flux pair $(\eta,q)$, by the condition 
\eqref{E_sign} it follows that $\nu^+=\delta_{\bar u}$. In particular the dissipation on $R^\infty$ is 0 and this concludes the proof.
\end{proof}

For every nontrivial $I\in \mathcal L_f$ and $t>0$ let $L(t,I)$ be the set of points $x\in \R$ for which $(t,x)$ is a Lebesgue point for $\nu$,
$\supp \nu_{t,x}\subset I$ and $\nu_{t,x}$ is not a Dirac delta. 
By the previous analysis it follows that for every nontrivial $I\in \mathcal L_f$ there exists  

\begin{equation*}
L(0,I):=\lim_{t\rightarrow 0}L(t,I) \quad \mbox{in }L^1.
\end{equation*}
Denote by $L(0)$ the union of $L(0,I)$ for $I\in\mathcal L_f$ nontrivial and let $D(0)=\R\setminus L(0)$.

In the following statement we summarize the results on concentration of entropy dissipation obtained in this section.
\begin{theorem}\label{P_conc}
Let $\nu$ be an entropy mv solution with a complete family of boundaries. Then for every entropy-entropy flux pair $(\eta,q)$ the dissipation
measure $\mu=\langle \nu,\eta\rangle_t+\langle \nu,q\rangle_x$ can be decomposed as $\mu=\mu_\diff + \mu_\jump$ where
\begin{enumerate}
\item $\mu_\jump$ is concentrated on $J$,
\item the image ${P_0}_\sharp \mu_\diff\ll \mathcal L^1$ and ${P_0}_\sharp \mu_\diff (D(0))=0$.
\end{enumerate}
\end{theorem}
\begin{remark}
If $\nu$ is a Dirac entropy solution then $D(0)=\R$ therefore Theorem \ref{T_main_theorem_conc} immediately follows from this result .
\end{remark}

\section{Initial data}\label{S_init}
We show that a mv entropy solution endowed with a compete family of boundaries assumes the initial datum in a strong sense.
The fact that the solution has a complete family of boundaries is used in the following lemma.

\begin{lemma}\label{L_mv_const}
Let $\bar \nu$ be a constant Young measure on $\R$ and let $\nu$ be a mv entropy solution with a complete family of boundaries on $\R^+\times \R$
such that for all entropy-entropy flux pairs
$(\eta,q)$ and $\varphi\in C^\infty_c([0,+\infty)\times \R)$
\begin{equation*}
\int_{\R^+\times \R}\left( \langle \nu,\eta\rangle \varphi_t+\langle \nu,q\rangle\varphi_x \right)dxdt + \int_\R\langle \bar\nu, \eta\rangle\varphi(0,x) dx = 0.
\end{equation*}
Then $\supp \bar \nu \subset I$ for some $I\in \mathcal L_f$ and $\nu_{t,x}=\bar \nu$ for $\mathcal L^2$-a.e. $(t,x)\in \R^+\times\R$.
\end{lemma}
\begin{proof}
Since there is no dissipation, for every $(\gamma,w)\in \K$, for every $t\in (0,T(\gamma,w))$ we have $u\in I^-=I^+$, where
$I^\pm$ are given by Proposition \ref{P_struct_J}. In particular $\gamma$ has constant speed $f'(u)$ in $(0,T(\gamma,w))$ and 
$\langle \nu, f'\rangle$ is continuous in $\R^+\times \R$.

We claim that every $\gamma\in \K_\gamma$ has constant speed in $(0,+\infty)$. Fix a positive time $T$ and for each $x\in \R$ let 
$(\gamma_x,w_x)\in \K$ be such that $\gamma_x(T)=x$ and $T(\gamma_x,w_x)\ge T$. The velocity $\dot\gamma_x$ is continuous in $x$ thanks to \eqref{E_relation}, therefore 
for every $t\in (0,T)$ we have 
\begin{equation*}
\bigcup_{x\in \R}\gamma_x(t)=\R. 
\end{equation*}
By arbitrariness of $T$ we have the claim.

By the dissipation formula \eqref{E_diss_form} we know that on each straight line $\nu$ is constant and therefore the initial condition implies that $\nu_{t,x}=\bar \nu$ 
for $\mathcal L^2$-a.e. $(t,x)\in \R^+\times \R$. In particular the curves $\gamma \in K_\gamma$ are parallel.
\end{proof}

\begin{lemma}\label{L_const_blow}
Let $\nu$ be a mv entropy solution with a complete family of boundaries.
For $\mathcal L^1$-a.e. $x\in \R$ the blow-up of $\nu$ about $(0,x)$ is the constant Young measure $\bar \nu$, where $\nu_{0,x}^+=\bar \nu$ is the
Lebesgue value of the trace $\nu_0^+$.
\end{lemma}
\begin{proof}
We first observe that for $\mathcal L^1$-a.e. $x\in \R$ the blow-up at $(0,x)$ is a mv entropy solution with a complete family of boundaries which does
not dissipate any entropy: in fact with a standard application of Vitali covering theorem it is possible to prove that for $\mathcal L^1$-a.e. $x\in \R$
\begin{equation}\label{E_no_diss}
\lim_{\e\rightarrow 0} \frac{\mu(B_{0,x}(\e)\cap \{t>0\})}{\e}=0
\end{equation}
and this implies the claim.

Consider a point $x\in \R$ such that \eqref{E_no_diss} holds and $x$ is a Lebesgue point of the trace $\nu^+_0$ with value $\bar \nu$.
With the notation introduced in Definition \ref{D_blow}, for all entropy-entropy flux pair $(\eta,q)$ and every test function $\varphi \in C_c^\infty([0,+\infty)\times \R)$
\begin{equation*}
\int_{\R^+\times \R}\left( \langle \nu^\e,\eta\rangle \varphi_t+\langle \nu^\e,q\rangle\varphi_x \right)dxdt + 
\int_\R\langle \nu^{\e+}_0, \eta\rangle\varphi(0,x) dx =
\int_{\R^+\times\R}\varphi d \mu^\e. 
\end{equation*}
Passing to the limit as $\e\rightarrow 0$ we get
\begin{equation*}
\int_{\R^+\times \R}\left( \langle \nu^0,\eta\rangle \varphi_t+\langle \nu^0,q\rangle\varphi_x \right)dxdt + \int_\R\langle \bar\nu, \eta\rangle\varphi(0,x) dx = 0.
\end{equation*}
By Lemma \ref{L_mv_const} this concludes the proof.
\end{proof}

Let $\mathfrak d$ be a bounded distance on $\mathcal P([-M,M])$ which induces the weak topology, for example the Wasserstein distance $W_2$.

\begin{lemma}\label{L_init_average}
The initial datum is assumed in the following sense: for every $L>0$
\begin{equation*}
\lim_{T\rightarrow 0}\frac{1}{T}\int_0^T\int_{-L}^L \mathfrak d(\nu_{t,x},\nu_{0,x}^+)dxdt =0,
\end{equation*}
where $\nu^+_{0,x}$ is the trace at $t=0$ of $\nu$.
\end{lemma}
\begin{proof}
By Lemma \ref{L_const_blow} and Egorov theorem for every $\e>0$ there exists $A_\e \subset [-L,L]$ and $\bar r>0$ such that $\mathcal L^1([-L,L]\setminus A_\e) <\e$ and for all $x\in A_\e$, $r\in (0,\bar r)$ 
\begin{equation*}
\frac{1}{2r^2}\int_0^r\int_{x-r}^{x+r} \mathfrak d(\nu_{t,x'},\nu_{0,x}^+)dx'dt < \e \qquad \mbox{and} \qquad \frac{1}{2r}\int_{x-r}^{x+r}\mathfrak d(\nu^+_{0,x'},\nu^+_{0,x})dx' <\e.
\end{equation*}
It is easy to see that for every $r\in (0,\bar r)$ it is possible to choose $x_1<\ldots<x_N$ in $A_\e$ such that every $x\in [-L,L]$ belongs to at most two of
the intervals $(x_i-r,x_i+r)$ and 
\begin{equation*}
\mathcal L^1\bigg([-L,L]\setminus \bigcup_{i=1}^{N}(x_i-r,x_i+r)\bigg)<\e.
\end{equation*}
Therefore
\begin{equation*}
\begin{split}
\frac{1}{r}\int_0^r\int_{-L}^L \mathfrak d(\nu_{t,x},\nu_{0,x}^+)dxdt =&~ 
\frac{1}{r}\int_0^r\left(\int_{A_\e} \mathfrak d(\nu_{t,x},\nu_{0,x}^+)dx + \int_{[-L,L]\setminus A_\e} \mathfrak d(\nu_{t,x},\nu_{0,x}^+)dx  \right)dt \\
\le &~\frac{1}{r}\sum_{i=1}^N\int_0^r\int_{x_i-r}^{x_i+r}\left[\mathfrak d(\nu_{t,x},\nu_{0,x_i}^+) + \mathfrak d(\nu_{0,x_i}^+,\nu_{0,x}^+)\right] dx dt \\
&~+ \bar{M}\mathcal L^1([-L,L]\setminus A_\e) \\
\le &~ 8L\e + \bar M\e,
\end{split}
\end{equation*}
where $\bar M$ is the supremum of $d$ in $\mathcal P([-M,M])^2$, and this concludes the proof.
\end{proof}

\begin{remark}\label{R_Chen}
The lemma above implies that there exists a sequence $t_n\rightarrow 0$ such that for every $L>0$
\begin{equation*}
\lim_{n\rightarrow \infty} \int_{-L}^L \mathfrak d(\nu_{t_n,x},\nu^+_{0,x})dx =0.
\end{equation*}
In particular we can deduce from Proposition \ref{P_struct_Jc} that for $\mathcal L^1$-a.e. $x\in \R$ there exists $I\in \mathcal L_f$ such that 
\begin{equation*}
\supp\, \nu^+_{0,x}\subset I
\end{equation*}
and if $\nu$ is a Dirac entropy solution then 
\begin{equation*}
\nu^+_{0,x} = \delta_{u_0(x)}
\end{equation*}
for some $u_0\in L^\infty$ and $\nu$ represents the unique entropy solution of \eqref{E_cl} with initial datum $u_0$.
\end{remark}

\begin{proposition}\label{P_init}
Let $\nu$ a mv solution with a complete family of boundaries. Then for every $L>0$
\begin{equation*}
\lim_{t\rightarrow 0} \int_{-L}^L \mathfrak d(\nu_{t,x},\nu^+_{0,x})\, dx =0,
\end{equation*}
 where $\nu_{t,x} = \nu^+_{t,x}$ is continuous from the right.
\end{proposition}

\begin{proof}
We use the same notation introduced at the end of the previous section. We prove separately the convergence in $\tilde L=L(0)\cap [-L,L]$ and $\tilde D= D(0)\cap [-L,L]$.
For every nontrivial $I\in \mathcal L_f$,  for $\mathcal L^1$-a.e. $x \in L(0,I)$ there exists the limit 
\begin{equation*}
\nu^+_{0,x}=\lim_{t\rightarrow 0}\nu_{t,x+f'(I)t},
\end{equation*}
because by entropy dissipation  the function $t \mapsto \nu_{t,x+f'(I)t}$ is $\BV_t$ for $\mathcal L^1$-a.e. $x \in L(0,I)$ when tested with $\mathcal C^2$ functions.
Since translations are continuous in $L^1$ it follows that
\begin{equation*}
\lim_{t\rightarrow 0}\int_{\tilde L} \mathfrak d(\nu_{t,x},\nu^+_{0,x})dx =0.
\end{equation*}
Hence it remains to prove the convergence on $D(0)$. It is sufficient to prove the claim with $\mathfrak d$ equal to the Wasserstein distance.
For every $x\in D(0)$, let $\nu_{0,x}^+=\delta_{u_0(x)}$ and consider a sequence $t_n\rightarrow 0$. 
We already know that $\nu_{t_n}\rightarrow \nu_0^+$ in the sense of Young measures by Proposition \ref{P_traces}: in particular this implies
\begin{equation*}
\begin{split}
\lim_{n\rightarrow \infty}\int_{\tilde D}\big\langle (w-u_0(x))^2,\nu_{t_n,x}\big\rangle dx=&~  
\lim_{n\rightarrow \infty}\left(\int_{\tilde D} \langle w^2,\nu_{t_n,x}\rangle dx + \int_{\tilde D}u_0^2(x)dx - 2\int_{\tilde D}u_0(x)\langle w,\nu_{t_n,x}\rangle dx\right) \\
=&~ 0
\end{split}
\end{equation*}
and this concludes the proof.
\end{proof}

We conclude with this extension of \cite{MR1771520}.
\begin{corollary}\label{C_Chen}
Suppose that $u$ is an entropy solution of $\eqref{E_cl}$ in the open set $\R^+\times \R$ and suppose that the initial datum is attained weakly* in 
$L^\infty$: for every sequence $t_n\rightarrow 0^+$
\begin{equation*}
u(t_n)\rightharpoonup u_0 \quad w^*-L^\infty.
\end{equation*}
Then the initial datum is attained in a strong sense: for every sequence $t_n\rightarrow 0^+$
\begin{equation*}
u(t_n)\rightarrow u_0 \quad s-L^1_\loc.
\end{equation*}
\end{corollary}
\begin{proof}
By Remark \ref{R_Chen} it is sufficient to show that $u$ has a complete family of boundaries. In order to prove it, we construct a sequence of
entropy solutions that converges to the Dirac solution $\delta_u$ in the sense of Young measures.

Assume for simplicity $u$ compactly supported.
Observe that since $\| u(t)\|_{L^2(\R)}$ is non increasing with respect to $t$ and $u:\R^+\rightarrow L^1(\R)$ is weakly continuous, $u$ can have at most
countably many discontinuity points with respect the strong topology $s-L^1$. 
Consider the entropy solutions $u_n$ with initial datum $u(t_n)$ for a sequence $t_n\rightarrow 0$ of strong continuity points of $u$. Then by Kruzkov theorem,
$u_n(t)=u(t+t_n)$ and in particular this implies that $u:(0,+\infty)\rightarrow L^1(\R)$ is strongly continuous and it has a complete family of boundaries. Therefore by Proposition \ref{P_init}, $\delta_{u_n}\rightarrow \delta_u$ in the sense of Young measures and this concludes the proof.
\end{proof}

\section{Lagrangian representation and counterexamples}\label{S_Lagr}

\subsection{Lagrangian representation}
Here we deduce the existence of a Lagrangian representation.
\begin{proposition}\label{P_lagr_repr}
Let $\nu$ be a mv entropy solution with a complete family of boundaries. Then there exists a couple of functions $(\X,\U)$ such that
\begin{enumerate}
\item $\X:[0,+\infty)\times \R\rightarrow \R$ is continuous, $t\mapsto\X(t,y)$ is Lipschitz for every $y$ and $y\mapsto \X(t,y)$ is non-decreasing for every $t$;
\item $\U\in L^\infty(\R)$;
\item there exists a representative of $\nu$ such that for every $(t,x)\in \R^+\times \R\setminus J$
\begin{equation*}
\langle \nu_{t,x},\I\rangle =  \bar u, \quad \mbox{where } \bar u=\U(\X(t)^{-1}(x));
\end{equation*}
\item the flow $\X$ satisfies the characteristic equation: for every $y\in \R$ for $\mathcal L^1$-a.e. $t>0$ it holds
\begin{equation*}
D_t \X(t,y)=
\begin{cases}
f'(I^+_t) & \mbox{if }I^+_t=I^-_t, \\
\displaystyle{\frac{\langle \nu^+_t,f\rangle - \langle \nu^-_t,f\rangle}{\langle \nu^+_t,\Id\rangle - \langle \nu^-_t,\Id\rangle}} &  \mbox{if }I^+_t\ne I^-_t,
\end{cases}
\end{equation*}
where $I^\pm_t\in\mathcal L_f$ contains the support of the trace $\nu^\pm$ at the point $(t,\X(t,y))$ from the left and the right of $\X(\cdot,y)$ 
(see Remark \ref{R_traces}).
\end{enumerate}
\end{proposition}
\begin{proof}
Since $\K_\gamma$ is a closed monotone family of Lipschitz curves that covers the whole $\R^+\times \R$, 
there exists a function $\X$ as in the statement such that for every $\gamma\in \K_\gamma$ there exists a unique $y\in \R$ for which 
$\gamma_y(t)=\X(t,y)$ in $(0,+\infty)$. So we only need to check that $\langle \nu_{t,\gamma_y(t)},\I\rangle$ is constant for $t$ such that 
$(t,\gamma_y(t))\in\R^+\times \R\setminus J$. Denote this set of times by $T_y$.

Since $(\gamma_y,\langle \nu_{t,\gamma_y(t)},\I\rangle)$ is an admissible boundary in $(0,t)$ for every $t\in T_y$, we have that there exists 
$I_y\in\mathcal L_f$
such that $\langle \nu_{t,\gamma_y(t)},\I\rangle\in I_y$ for every $t\in T_y$. This in particular implies the claim for all $y$ such that $I_y=\{u\}$ for 
some $u\in \R$, therefore it suffices to consider the set where $\nu$ takes values in a linearly degenerate component of the flux.
In this case Lemma \ref{L_aver_const} implies that the claim is true in $(0,T_1(\gamma_y))$ and from Remark \ref{R_V_Riem} it follows that 
$(t,\gamma_y(t))\in J$ for every $t>T_1(\gamma_y)$.
\end{proof}

\subsection{Counterexample 1}
Consider an entropy solution $u$ of \eqref{E_cl}.
Observe that at time 0 there is a set $J_0\subset \R$ at most countable such that  every point of $J_0$ is the starting point of two different curves.
For every $x\in \R\setminus J_0$ denote by $y_x=\X(0)^{-1}(x)$.
In Section \ref{S_struct} we saw that every $(t,x)\in\R^+\times \R$ either belongs to a segment starting from 0 or the rectifiable set $J$ or a 
domain of a Riemann problem with two boundaries.
Nevertheless it is in general not true that for $\mathcal L^1$-a.e. $x\in \R$ there exists $t_x>0$ such that $\partial_t\X(t,y_x)$ 
is constant in $(0,t_x)$. In particular it is not true that for $\mathcal L^1$-a.e. $x\in \R$ the value $u_0(x)$ is transported along a characteristic for a 
positive time.

The counterexample is the entropy solution $u$ of Burgers' equation
\begin{equation*}
u_t+\left(\frac{u^2}{2}\right)_x=0,
\end{equation*}
where the initial datum is the characteristic function of a Cantor set $C$
of positive measure, and for every $t>0$ the level set $\{u(t)=1\}$ has Lebesgue measure 0.

It is well-known that the function 
\begin{equation*}
U(t,x)=\int_{-\infty}^x u(t,z)dz
\end{equation*} 
is the viscosity solution of the Hamilton-Jacobi equation
\begin{equation*}
U_t+\left(\frac{U_x}{2}\right)^2=0.
\end{equation*}
In particular $U$ can be obtained by Lax formula:
\begin{equation}\label{E_Lax}
U(t,x) = \min_{y\in\R} \left\{ U(0,y) + \frac{|x-y|^2}{2t} \right\}.
\end{equation}
We provide an example of $C$ such that for $\mathcal L^1$-a.e. $y\in C$ there are no $(t,x)\in \R^+\times \R$ such that $y$ is the minimizer in
\eqref{E_Lax}.

\textbf{Claim}. Let $y$ be a point of density one for $C$. If there exists $(t,x)\in\R^+\times \R$ such that $y$ is a minimizer in \eqref{E_Lax}, then for every $\theta >0$,
\begin{equation}\label{E_square_density}
\frac{1}{\theta^2}\int_{y}^{y+\theta}\chi_{C^c}(z)dz \le \frac{1}{2t}.
\end{equation}
{\it Proof of the claim}. Let $\theta\in \R$, by minimality
\begin{equation*}
U(0,y)+\frac{|x-y|^2}{2t}\le U(0,y+\theta)+\frac{|x-y-\theta|^2}{2t}
\end{equation*}
therefore 
\begin{equation}\label{E_minimality}
U(0,y)-U(0,y+\theta)\le \frac{|x-y-\theta|^2}{2t}-\frac{|x-y|^2}{2t} = -\frac{(x-y)\theta}{t}+\frac{\theta^2}{2t}.
\end{equation}
Since \eqref{E_minimality} holds for every $\theta$ positive and negative and $U(0)$ has derivative equal to 1 at $y$ we get
\begin{equation*}
\frac{x-y}{t}=1, \qquad x=y+t. 
\end{equation*}
We get \eqref{E_square_density} from \eqref{E_minimality} observing that 
\begin{equation*}
\int_y^{y+\theta}\chi_{C^c}(z)dz=U(0,y)-U(0,y+\theta) +\theta.
\end{equation*}
The last step is the construction of a Cantor set $C$ of positive measure such that for every $y\in C$
\begin{equation*}
\limsup_{\theta\rightarrow 0}\frac{1}{\theta^2}\int_{y}^{y+\theta}\chi_{C^c}(z)dz =+\infty.
\end{equation*}

On the interval $C_0=[0,2]$ consider the standard Cantor construction where $C_n$ is obtained from $C_{n-1}$ removing the middle interval 
of size $3^{-n}$ in each connected component of $C_{n-1}$.
Then 
\begin{equation*}
C=\bigcap_{n\in \N}C_n
\end{equation*}
has Lebesgue measure equal to 1. Fix $\bar y\in C$, for every $n\in N$ let $y_n$ be the minimal $y >\bar y$ such that $y$ is the
left endpoint of a connected component of $C_n$.
Since the length of every connected component of $C_n$ is bounded by $2^{-n+1}$, by direct checking
\begin{equation*}
\frac{1}{(y_n-y)^2}\int_y^{y_n}\chi_{C^c}(z)dz \ge \frac{3^{-n}}{(2^{-n+1}+3^{-n})^2}\rightarrow +\infty.
\end{equation*}

\begin{figure}
\centering
\begin{minipage}{0.5\columnwidth}
\def\svgwidth{\columnwidth}
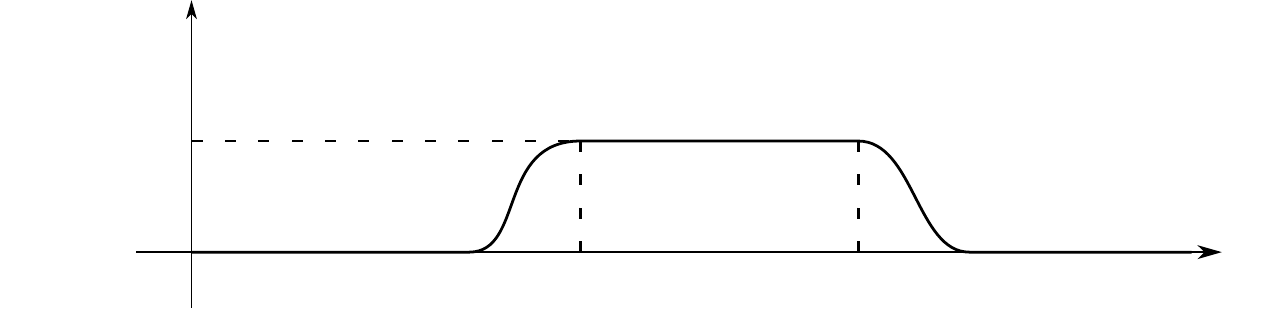
\caption{Flux $f^n_{a,L}$.}\label{F_flux_n}
\end{minipage}
\quad
\begin{minipage}{0.4\columnwidth}
\def\svgwidth{\columnwidth}
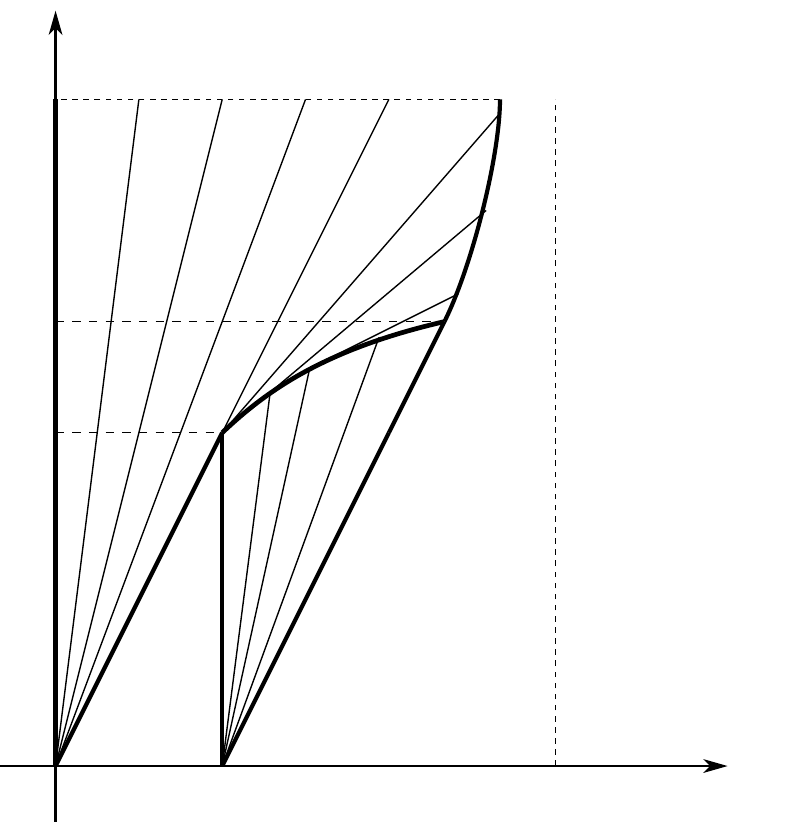
\caption{Solution to $u_0=2L\chi_{[0,d]}$ with flux $f^n_{a,L}$.}\label{F_sol_block}
\end{minipage}
\end{figure}


\begin{figure}
\centering
\begin{minipage}{0.45 \columnwidth}
\def\svgwidth{\columnwidth}
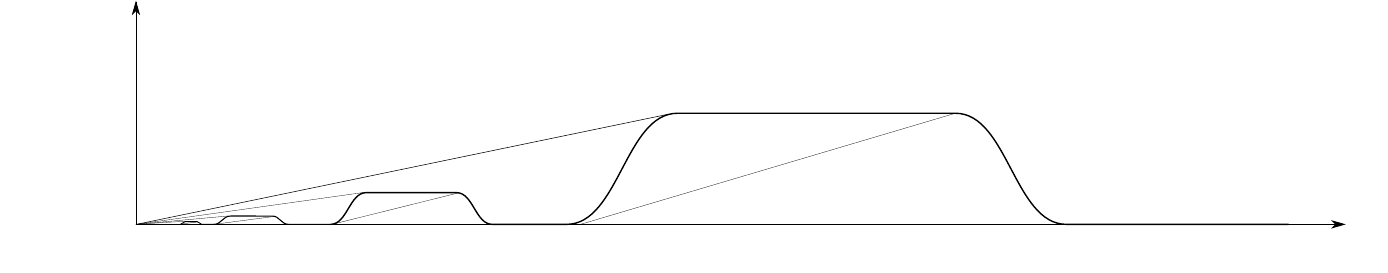
\caption{Flux $f$.}\label{F_flux_f}
\end{minipage}
\quad
\begin{minipage}{0.45\columnwidth}
\def\svgwidth{\columnwidth}
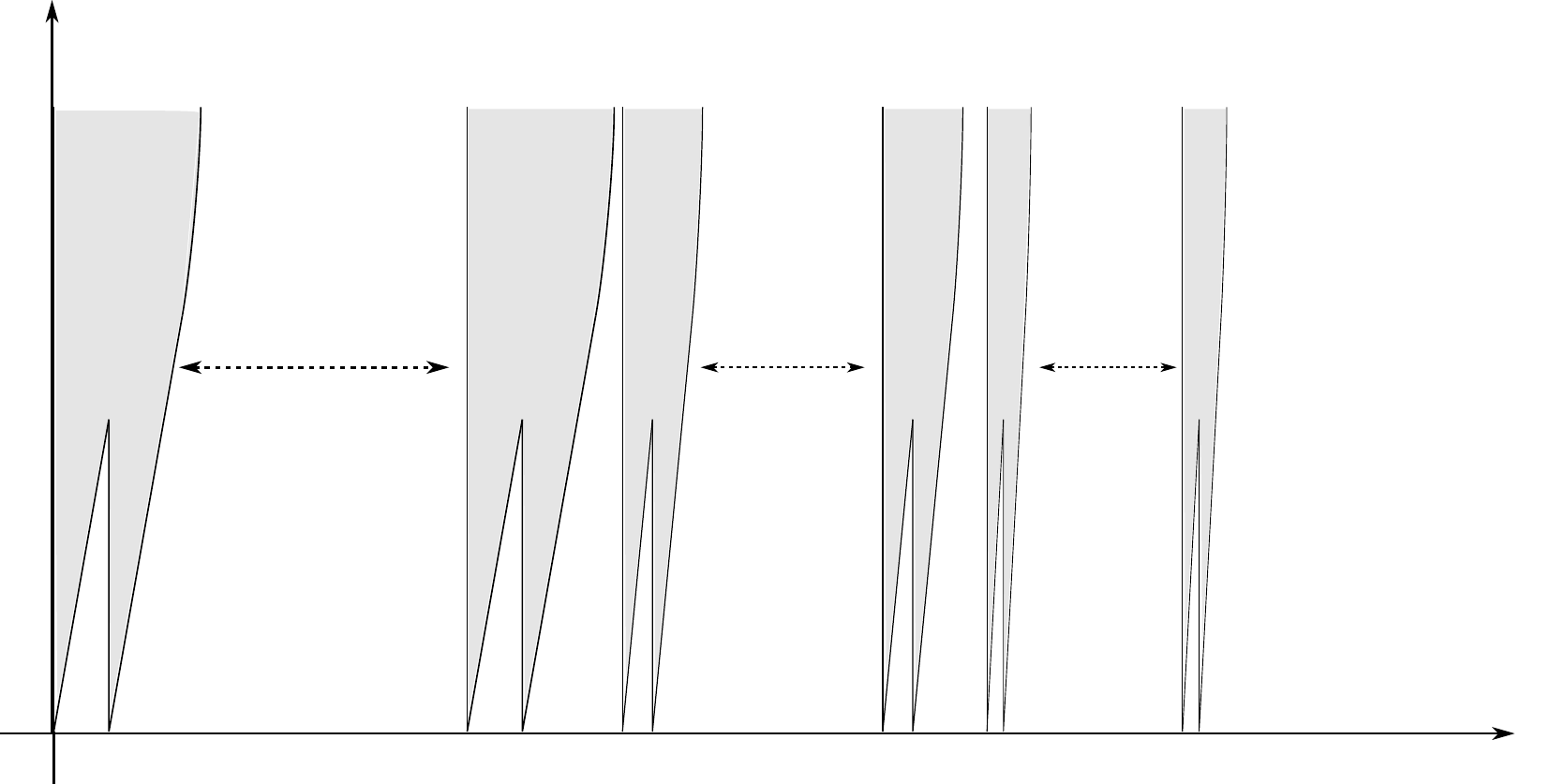
\caption{Solution with $\TV f'\circ u = +\infty$.}\label{F_sol_final}
\end{minipage}
\end{figure}

\subsection{Counterexample 2}
\label{Ss_counter_2}

Here we present an example of an $L^\infty$ entropy solution $u$ of \eqref{E_cl} such that $f'\circ u$ has no bounded variation locally in 
$\R^+\times \R$.

{\it The building block}. 
Consider a function $g\in C^\infty([-1,1])$ such that 
\begin{enumerate}
\item $g(-1)=0, \quad g(0)= \frac{1}{2}, \quad g(1)=1$;
\item $g$ is convex in $[-1,0]$ and concave in $[0,1]$;
\item $g'(0)=1$;
\item all derivatives vanishes at the points $-1$ and $1$;
\item $g-\frac{1}{2}$ is odd.
\end{enumerate}
Let $a,L>0$ and $n\in \N$ be parameters such that $3a\le L$ and consider the smooth flow $f^n_{a,L}$ as in Figure \ref{F_flux_n}:
\begin{equation*}
f^n_{a,L}(u)= \begin{cases}
0 & \mbox{if }u\le L-a, \\
a^ng\left(\frac{u-L}{a}\right) & \mbox{if }L-a<u\le L+a, \\
a^n  & \mbox{if }L+a<u\le 2L, \\
a^n g\left(\frac{L+a-u}{a}\right)  & \mbox{if }2L<u\le 2L+2a, \\
0  & \mbox{if }u>2L+2a.
\end{cases}
\end{equation*}

The initial datum is 
\begin{equation*}
u_0 = 2L \chi_{[0,d]},
\end{equation*}
where $d>0$ will be fixed below.
For $t$ small the solution is obtained solving separately the two Riemann problems (see Figure \ref{F_sol_block}): 
the problem in 0 has a first shock $[0,L-a]$ of velocity 0, then a rarefaction from $L-a$ to $L-b$ for some $b\in (0,a)$ and a second shock $[L-b,2L]$.
Let
\begin{equation*}
d= \frac{f(2L)-f(L-b)}{L+b}
\end{equation*}
be equal to the velocity of the second shock in 0. Since $b\in (0,a)$ we have 
\begin{equation*}
d\in \bigg( \frac{a^n}{L+a}, \frac{a^n}{L} \bigg).
\end{equation*}
The solution of the second
Riemann problem has the same structure. It follows that there is a cancellation in $(1,d)$ from which it starts the shock number 5 of Figure \ref{F_sol_block} . Let $t_1$ be the time
for which the shock 5 collides with the shock 4. Since the shock 4 has constant velocity equal to $d$ and the shock 5 has velocity 
$v(t)\in (\frac{a^{n-1}}{2}, a^{n-1})$, we have 
\begin{equation*}
t_1>1+\frac{2a}{L-2a}.
\end{equation*}
For every $t\in (t_1,2)$ the maximal velocity $v_{\max} (t)$ at time $t$ is the velocity of characteristics which enters in shock 6 at time $t$:
in particular $v_{\max} (t)(t-1)\ge d \ge\frac{a^n}{L+a}$.
Moreover observe that the solution $u(t)$ has support contained in $[0, 3d]\subset [0,\frac{3a^n}{L}]$ for every $t\in [0,2]$.
The estimate on the total variation is 
\begin{equation*}
\int_1^2\int_0^{\frac{3a^n}{L}}|D_x f'(u(t,x))|dx dt \ge \int_{1+\frac{2a}{L-2a}}^2 \frac{a^n}{(L+a)(t-1)}dt = \frac{a^n}{L+a}\log \left(\frac{L-2a}{2a}\right).
\end{equation*}

The point is that for $a \ll L$ in an interval of length of the order $\frac{a^n}{L}$ the total variation is of the order of 
$\frac{a^n}{L}\log \left(\frac{L}{a}\right)$.

{\it The general case}. Consider the flux (Figure \ref{F_flux_f})
\begin{equation*}
f=\sum_{n=1}^\infty f^n_{a_n,L_n}.
\end{equation*}
Observe that if $4L_{n+1}\le L_n$ the supports of $f^n_{a_n,L_n}$ are disjoint and $f\in C^\infty_c(\R)$.
The initial datum is obtained by placing side by side $N_1$ initial data of the form $2L_1\chi_{[0,d_1]}$, $N_2$ initial data of the form $2L_2\chi_{[0,d_2]}$ and so on, see Figure \ref{F_sol_final}. 

The condition
\begin{equation}\label{E_no_inter}
a_{n+1}^{n+1} < \frac{L_n}{a_n^n}L_{n-1}
\end{equation}
guarantees that for every $n$ the solution with initial datum $2L_n\chi_{[0,d_n]}$ with flux $f$ is the same as the solution with the same initial datum and flux
$f^n_{a_n,L_n}$.
In order to have infinite total variation in an interval of finite length, 
it suffices to provide three sequences $(a_n)_{n\in \N}$, $(L_n)_{n\in \N}$ and $(N_n)_{n\in \N}$ such that \eqref{E_no_inter} holds, $3a_n \le L_n$, $4L_{n+1}\le L_n$,
\begin{equation*}
\sum_{n=1}^\infty N_n\frac{a_n^n}{L_n}<+\infty \qquad \mbox{and}\qquad 
\sum_{n=1}^\infty N_n\frac{a_n^n}{L_n+a_n}\log  \left(\frac{L_n-2a_n}{2a_n}\right)=+\infty.
\end{equation*}
For example consider $a_n=4^{-n}$, $L_n=3\cdot 4^{-n}$ and $N_n$ equal to the integer part of $\frac{L_n}{n^2a_n^n}$.

%
%
%

%

%
%
%

\bibliographystyle{siam}
\bibliography{biblio}

\end{document}